\documentclass[10pt]{amsart}
\usepackage{calrsfs}

\usepackage{longtable,booktabs}
\usepackage{tabularx}

\usepackage{graphicx}
\usepackage{hyperref}
\usepackage{amsthm}
\usepackage{amssymb}
\usepackage{multirow}
\usepackage{tikz-cd}
\usepackage{amsmath}
\usepackage{todonotes}
\usepackage{amsbsy}
\usepackage[all]{xy}
\usepackage{color, colortbl}
\definecolor{LightCyan}{rgb}{0.88,1,1}
\definecolor{Gray}{gray}{0.9}
\definecolor{aquamarine}{rgb}{0.5, 1.0, 0.83}
\definecolor{ashgrey}{rgb}{0.7, 0.75, 0.71}
\usepackage{changes}
\definechangesauthor[color=orange,name={Aristides Kontogeorgis}]{AK}
\definechangesauthor[color=blue,name={Alex Terezakis}]{AT}

\newtheorem{theorem}{Theorem}
\newtheorem{lemma}[theorem]{Lemma}
\newtheorem{corollary}[theorem]{Corollary}
\newtheorem{proposition}[theorem]{Proposition}
\theoremstyle{definition}

\newtheorem{remark}[theorem]{Remark}
\newtheorem{definition}[theorem]{Definition}

\DeclareMathOperator{\A}{Aut}
\DeclareMathOperator{\Aut}{Aut}
\DeclareMathOperator\GL{GL}
\DeclareMathOperator\Sym{Sym}

\DeclareMathOperator{\Spe}{Spec}

\newcommand{\Dgl}{{D_{\rm gl}}}

\newcommand{\HomC}{\mathcal{H}\!\mathit{om}}

\DeclareMathOperator{\Ima}{Im}

\renewcommand{\mod}{{\;\rm mod}}

\newcommand{\Rgl}{R}

\newcommand{\s}{\sigma}

\newcommand{\defeq}{\mathrel{\vcenter{\baselineskip0.5ex \lineskiplimit0pt
                     \hbox{\scriptsize.}\hbox{\scriptsize.}}}%
                     =}

\date{\today}

\title[Deformations with automorphisms]{The canonical ideal and the deformation theory of curves with automorphisms}

\author[A. Kontogeorgis]{Aristides Kontogeorgis}
\address{Department of Mathematics, National and Kapodistrian  University of Athens
Pane\-pist\-imioupolis, 15784 Athens, Greece}
\email{kontogar@math.uoa.gr}

\author[A. Terezakis]{Alexios Terezakis }
\address{Department of Mathematics, National and Kapodistrian University of Athens\\
Panepistimioupolis, 15784 Athens, Greece}
\email{aleksistere@math.uoa.gr}

\date \today

\makeatletter
\newcommand{\aprod}{\mathop{\operator@font \hbox{\Large$\ast$}}}
\makeatother


\begin{document}

\begin{abstract}
The deformation theory of curves is studied by using the canonical ideal. The deformation  problem of curves with automorphisms is reduced to a deformation problem of linear representations. 
\end{abstract}
\maketitle


\section{Introduction}



The deformation theory of curves with automorphisms is an important generalization of the classical deformation theory of curves. 
This theory is related to the lifting problem of curves with automorphisms, since one can consider liftings from characteristic $p>0$ to characteristic zero in terms of a sequence of local Artin-rings.

J. Bertin and A.  M\'ezard in \cite{Be-Me}, following Schlessinger's  \cite{Sch} approach introduced a deformation functor $D_{\mathrm{gl}}$  and studied it 
in terms of  Grothendieck's equivariant cohomology theory \cite{GroTo}.
In Schlessinger's approach to deformation theory, we want to know the tangent space to the deformation functor $D_{\mathrm{gl}}(k[\epsilon])$ and the possible obstructions to lift a deformation over an Artin local ring $\Gamma$ to a small extension $\Gamma' \rightarrow \Gamma$. The reader who is not familiar with deformation theory is referred to section \ref{sec:DeformationTheory}
for terminology and references to the literature.
 The tangent space  of the global deformation functor $\Dgl(k[\epsilon])$ can be identified  as Grothendieck's equivariant cohomology group
$H^1(G,X,\mathcal{T}_X)$, which is known to be equal to the invariant space $H^1(X,\mathcal{T}_X)^G$.
Moreover, a  local local-global theorem is known, which can be expressed in terms of 
the short exact sequence:
\begin{equation}
\label{BeME-lg}
\xymatrix@C-4pt@R-12pt{
    0 \ar[r] &
    H^1(X/G, \pi_*^G(\mathcal{T}_X)) \ar[r]&
    H^1(G,X,\mathcal{T}_X) \ar[r] & 
    H^0(X/G, R^1 \pi_*^G (\mathcal{T}_X)) \ar[d]^-{\rotatebox{90}{$\cong$}} \ar[r] 
    & 0 
    \\
   & & & \displaystyle\bigoplus_{i=1}^r  H^1\left(G_{x_i}, \widehat{\mathcal{T}}_{X,x_i} \right) &
}
\end{equation}
The lifting obstruction can be seen as an element in 
\[
H^2(G,X, \mathcal{T}_X) \cong \bigoplus_{i=1}^r 
H^2\left(G_{x_i},\widehat{\mathcal{T}}_{X,x_i}\right). 
\]
In the above equations $x_1,\ldots,x_r \in X$ are the ramified points, $G_{x_i}$ are the corresponding isotropy groups and $\widehat{\mathcal{T}}_{X,x_i}$ are the completed local tangent spaces, that is 
$\widehat{\mathcal{T}}_{X,x_i}= k[[t_i]] \frac{d}{dt_i}$, where $t_i$ is a local uniformizer at $x_i$. The space  $k[[t_i]] \frac{d}{dt_i}$
is seen as $G_{x_i}$-module by the adjoint action, see \cite[2.1]{CK}, \cite[1.5]{KontoANT}.
J. Bertin and A. M\'ezard reduced the computation of obstruction 
to the infinitesimal lifting problem of representations of the isotropy group $G_{x_i}$ to the difficult group $\Aut k[[t]]$. 
 In this article for a ring $\Gamma$, $\Aut \Gamma[[t]]$ denotes the group of continuous automorphisms of $\Gamma[[t]]$.

\bigskip
This article aims to give a new approach to the deformation theory of curves with automorphisms, which is not based on the deformation theory of representations on the subtle object $\Aut k[[t]]$, but on the deformation theory of the better understood general linear group. 
Our work is motivated by the 
problem of deforming and lifting curves with automorphisms, and is a part of a series of articles \cite{MR4130074}, \cite{MR4333646}, \cite{MR4194180}, \cite{MR4779377}, \cite{1901.08446} aiming to this goal.
More precisely
theorem \ref{th:main-lift} and proposition 
\ref{red-new-prot}
are used in \cite{kontogeorgis2023new} in order to provide a counterexample to the generalized Oort conjecture. 
The Oort conjecture states that every cyclic group $C_q$ of order 
$q=p^h$ is a local Oort group.
A local Oort group $G$ is a group such that for every representation $G \rightarrow \A(k[[t]])$, there exist an integrally closed domain $\Lambda$ 
contained in a field extension of the quotient field  $\mathrm{Frac}(W(k))$ of Witt vectors $W(k)$  
and a representation
\[
\tilde{\rho}: G \hookrightarrow \A(\Lambda[[T]]),
\]
such that if $t$ is the reduction of $T$, then the action of $G$ on $\Lambda[[T]]$ defined by $\tilde{\rho }$
reduces to the action of $G$ on $k[[t]]$ defined by $\rho $.
The Oort conjecture is recently proved by F. Pop \cite{MR3194816}, using the work of A. Obus and S. Wewers \cite{ObusWewers}. The generalized Oort conjecture, which was believed to be correct by experts of the field, states that the dihedral group $D_{p^h}$ of order $2p^h$, where p is an odd prime, is also a local Oort group. For more information the reader is referred to \cite{MR2441248}, \cite{MR2919977}, \cite{MR3591155}, \cite{Obus12}.

In this article we will restrict ourselves to curves that satisfy the mild assumptions of Petri's theorem   
\begin{theorem}[Petri's theorem]
\label{th:PetriTheorem-k}
Let $X$ be a non-singular, non-hyperelliptic curve  of genus $g\geq 3$, defined over an algebraically closed field. Let $\Omega_X$ be the  sheaf of differentials of $X$.   There is the following short exact sequence:
\[
0 \rightarrow I_X \rightarrow \Sym  H^0(X, \Omega_X) \rightarrow \bigoplus_{n=0}^\infty H^0(X, \Omega_X^{\otimes n})
\rightarrow 0,
\]
where $I_X$ is generated by elements of degree $2$ and $3$. Also if $X$ is not a non-singular quintic of genus $6$ or $X$ is not a trigonal curve, then $I_X$ is generated by elements of degree 2. 
\end{theorem}
For a proof of this theorem we refer to \cite{MR895152}, \cite{Saint-Donat73}. The ideal $I_X$ is called {\em the canonical ideal} and it is the homogeneous ideal of the embedded curve $X\rightarrow \mathbb{P}^{g-1}$.

For curves that satisfy the assumptions of Petri's theorem and their canonical ideal is generated by quadrics, we prove in section 
\ref{sec:RelPetriThm}
the following relative version of Petri's theorem

\begin{proposition}
\label{red-new-prot}
Let $A$ be a local Artin ring or the versal deformation ring $R$ of the deformation functor of curves, see section \ref{sec:RelPetriThm} for a definition of the ring $R$.  Let $f_1,\ldots,f_r \in \Sym H^0(X,\Omega_X)=k[\omega_1,\ldots,\omega_g]$ be quadratic polynomials which generate the canonical ideal $I_{X}$ of a curve $X$ defined over an algebraic closed field $k$. Any deformation 
$X_A$ is given by quadratic polynomials $\tilde{f}_1,\ldots,\tilde{f}_r \in \Sym  H^0(X_A,\Omega_{X_A/A})=A[W_1,\ldots,W_g]$, which reduce to $f_1,\ldots,f_r$ modulo the maximal ideal 
$\mathfrak{m}_A$ of $A$. 
\end{proposition}

{ 
\begin{definition}\label{1stDef}
We will denote by $S_A$ the symmetric algebra 
\[
    \Sym H^0(X_A,\Omega_{X_A/A})=A[\omega_1,\ldots,\omega_g]
\]
 and, specifically, in the case where $R=k$, we will simply denote it as $S$.
\end{definition}
}

This approach allows us to replace several constructions of Grothendieck's equivariant cohomology  in terms of linear algebra. Let us mention that 
in general, it is not so easy to perform explicit computations with equivariant Grothendieck cohomology groups and usually, spectral sequences or a complicated equivariant {\v{C}ech cohomology} is used, see \cite{BeMe2002}, \cite[sec.3]{KoJPAA06}. 

Let $i:X \rightarrow \mathbb{P}^{g-1}$ be the canonical embedding and
let $M_g(k)/\langle \mathbb{I}_g \rangle$ be the space of $g\times g$ matrices with coefficients in $k$, modulo the vector subspace of scalar multiples of the identity matrix.
In proposition \ref{prop12} we prove that elements $[f] \in H^1(X,\mathcal{T}_X)^G=D_{\mathrm{gl}} k[\epsilon]$ correspond to cohomology classes in $H^1(G,M_g(k)/\langle \mathbb{I}_g \rangle)$.

Furthermore, in our setting, the obstruction to liftings is reduced to an obstruction to the lifting of the linear canonical representation 
\begin{equation}
\label{rhodef}
\rho: G \rightarrow \GL \big( H^0(X,\Omega_X)  \big).
\end{equation}
Also we will give  a compatibility criterion involving the defining quadratic equations of our canonically embedded curve, namely in section \ref{sec:grpsDefs} we will prove the following:

\begin{theorem}
\label{th:main-lift}
Let $X \rightarrow \Spe k$ be a curve satisfying the assumptions of Petri's theorem and whose canonical ideal is generated by quadratic polynomials.
Let $X_A \rightarrow \Spe A$ be a deformation of $X$, where $A$ is a local ring with 
$A/\mathfrak{m}_A=k$. An automorphism $\sigma \in \A(X)$ can be lifted in an automorphism of $X_A$ if and only if the canonical ideal $I_{X_A}$ is left invariant under the action of $\sigma$. 

In particular consider an epimorphism $\Gamma'\rightarrow \Gamma\rightarrow 0$ of local Artin rings.
A deformation $x \in D_{\mathrm{gl}}(\Gamma)$ can be lifted to a deformation $x' \in  D_{\mathrm{gl}}(\Gamma')$ if and only if
the representation $\rho_\Gamma :G \rightarrow \GL_g(\Gamma)$ lifts to a representation  $\rho_{\Gamma'}:G \rightarrow \GL_g(\Gamma')$  and moreover there is a lifting $X_{\Gamma'}$ of the embedded deformation of $X_\Gamma$ which is  invariant under the lifted action of $\rho_{\Gamma'}$. 
\end{theorem}

\begin{remark}
The liftability of the representation $\rho$ is a strong condition. In proposition \ref{prop:lift-obsturctions} we give an example of a representation $\rho: G \rightarrow \GL_2(k)$, for a field $k$ of positive characteristic $p$, which can not be lifted to a representation 
$\tilde{\rho}:G \rightarrow \GL_2(R)$ for $R=W(k)[\zeta_{p^h}]$, meaning that a lifting in some small extension $R/\mathfrak{m}_R^{i+1} \rightarrow R/\mathfrak{m}_R^i$ is obstructed. Here $R$ denotes the Witt ring of $k$ with a primitive $p^h$ root of unity added, which has characteristic zero. In our counterexample $G=C_q \rtimes C_m$, $q=p^h$, $(m,p)=1$. 

In \cite{MR4779377} the authors give a necessary  condition for the lifting of a representation of $C_q \rtimes C_m$ from characteristic $p$ to characteristic zero. 
\end{remark}

\begin{remark}
One can always pass from the local lifting problem of $\rho:G \rightarrow \A \Gamma[[t]]$ to a global lifting problem, by considering the Harbater-Katz-Gabber (HKG for short) compactification $X$ of the local action. Then one can consider the criterion involving the linear representation $\rho:G \rightarrow \GL(H^0(X,\Omega_X))$. Notice that in \cite{MR4194180} the canonical ideal for {HKG-curves} is explicitly described.    

 \end{remark}

\begin{remark}
The invariance of the canonical ideal $I_{X_\Gamma}$ under the action of $G$  can be checked using Gauss elimination and echelon normal forms, see \cite[sec 2.2]{MR4333646}. 
\end{remark}
\begin{remark}\label{Rem5}
The canonical ideal $I_{X_{\Gamma}}$ is determined by $r$ quadratic polynomials which  form a $\Gamma[G]$-invariant $\Gamma$-submodule $V_\Gamma$ in the free $\Gamma$-module of symmetric $g\times g$ matrices with entries in $\Gamma$. When we pass from a deformation $x \in D_{\mathrm{gl}}(\Gamma)$ to 
a deformation $x' \in  D_{\mathrm{gl}}(\Gamma')$ we ask that the canonical ideal
 $I_{X_{\Gamma'}}$ is invariant under the lifted action given by the representation $\rho_{G'}:G\rightarrow \GL_{g}(\Gamma')$. 
 In definition \ref{def:T-action}.1 we will introduce an action $T$ on the vector space of symmetric $g\times g$ matrices, and the invariance of the canonical 
 ideal is equivalent to the invariance under the $T$-action  of the  $\Gamma'$-submodule  $V_{\Gamma'}$   generated by the quadratic polynomials generating the ideal  $I_{X'}$. 
{
Notice that Petri's theorem gives rise to the first steps of an $S$ resolution of the homogeneous ideal of the projective embedding of the curve, see eq. (\ref{gen-diagram}). Similarly a minimal set of quadratic generators of Petri's theorem are identified to $\mathrm{Tor}_1(k,I_X)$, since we have a free resolution 
\[
   \cdots \rightarrow F_i  \rightarrow \cdots  \rightarrow F_1 \rightarrow I_X \rightarrow 0,
\]
where $F_i= \oplus_\nu  m_{i,\nu} S$ is considered to be freely generated by the elements $m_{i,\nu}$. Using Nakayama's lemma one can show that the minimality of generators is equivalent to 
$\mathrm{Tor}_1(k,I_X) = k \otimes_S F_1$, \cite[prop. 1.7]{MR2103875}. 
The dimension $r$ of the space quadratic generators is identified to the Betti number $\beta_{1,2}$ which is equal to $\binom{g-2}{2}$, see \cite[prop. 9.5]{MR2103875}. 
}
 Therefore, we can write one more representation 
\begin{equation}
\label{rho1def1}
\rho^{(1)}: G \rightarrow \GL \big(  \mathrm{Tor}_1^S (k, I_X) \big).
\end{equation}
{ coming from the action of $G$ on the quadratic generators of $I_X$. For more information about the action of the automorphism group on a minimal free resolution of the homogeneous ring of a canonical embedded curve we refer to \cite{MR4333646}. }
Set $r=\binom{g-2}{2}$.
Liftings of 
the representations $\rho,\rho^{(1)}$ defined by eq. (\ref{rhodef}), (\ref{rho1def1}) 
 in $\GL_g(\Gamma)$ resp. $\GL_r(\Gamma)$ will be denoted by $\rho_{\Gamma}$ resp. $\rho^{(1)}_{\Gamma}$.

Notice that if the representation $\rho_\Gamma$ lifts to a representation 
$\rho_{\Gamma'}$ and moreover there is a lifting  $X_{\Gamma'}$ of the relative curve
 $X_\Gamma$ so that $X_{\Gamma'}$ has an ideal $I_{X_{\Gamma'}}$ which is $\rho_{\Gamma'}$ invariant, then the representation $\rho^{(1)}_\Gamma$ also lifts to a representation $\rho^{(1)}_{\Gamma'}$, see also \cite[prop. 5]{MR4333646}.

\end{remark}

The deformation theory of linear representations $\rho,\rho^{(1)}$ gives rise to cocycles 
 $D_\sigma$, $D^{(1)}_{\sigma^{-1}}$  in $H^1(G,M_g(k))$, $H^1(G,M_{\binom{g-2}{2}}(k))$,
while the deformation theory of curves with automorphisms 
introduces a cocycle $B_\sigma[f]$ corresponding to $[f] \in H^1(X,\mathcal{T}_X)^G$. 
We will introduce a compatibility condition in the section 
\ref{sec:AtangenSpaceCond}
among these cocycles, using the isomorphism 
\begin{align*}
\psi: M_g(k)/\langle \mathbb{I}_g \rangle 
&\stackrel{\cong}{\longrightarrow}
H^0(X,i^* \mathcal{T}_{\mathbb{P}^{g-1}}) \hookrightarrow
 \mathrm{Hom}_S (I_X, S/I_X) = H^0(X,\mathcal{N}_{X/\mathbb{P}^{g-1}})
\\
B & \longmapsto \psi_B
\end{align*}
defined in 
proposition \ref{psimapiso}. 
\begin{proposition}
\label{prop:4compat}
The following compatibility condition is satisfied
\begin{equation}
\label{eq:compat-cond}
\psi_{D_{\sigma}}-\psi_{B_\sigma[f]}=D_{\sigma^{-1}}^{(1)}.
\end{equation}
\end{proposition}

\bigskip
We will now describe the structure of this article.
In section \ref{sec:unifyReptheories}
we will present side by side the deformation theory of linear representations $\rho:G\rightarrow \GL(V)$ and the deformation theory of representations of the form $\rho: G \rightarrow  \A k[[t]]$. The deformation theory of linear representations is a better-understood object of study, see \cite{MR818915}, which played an important role in topology \cite{MR1404926} and also in the proof of Fermat's last theorem, see \cite{MazDef}. The deformation theory of representations in $ k[[t]]$ comes out from the study of local fields 
and it is related to the deformation problem of curves with automorphisms after the local global theory of J. Bertin and A. M\'ezard. There is also an increased interest related to the study of Nottingham groups and $ k[[t]]$, see \cite{CaminaRachel},\cite{wildthings},\cite{1901.08446}.

It seems that the similarities between these two deformation problems are known to the experts, see for example \cite[prop. 3.13]{OlsonMSRI}. For the convenience of the reader and in order to fix the notation, we also give a detailed explanation and comparison of these two deformation problems. 


In section \ref{sec:RelPetriThm} we revise the theory of relative canonical ideals and the work of the first author together with H. Charalambous and K. Karagiannis \cite{1905.05545} aiming at the deformation problem of curves with automorphisms. More precisely a relative version of Petri's theorem is proved, which implies that the relative canonical ideal is generated by quadratic polynomials.

In section \ref{sec:grpsDefs} we study both the obstruction and the tangent space problem of the deformation theory of curves with automorphisms using the relative canonical ideal point of view. In this section theorem \ref{th:main-lift} is proved. 

\smallskip

\noindent {\bf Acknowledgements}
 The research project is implemented in the framework of H.F.R.I. Call “Basic research Financing Horizontal support of all Sciences)” under the National Recovery and Resilience Plan “Greece 2.0” funded by the European Union Next Generation EU, H.F.R.I.  
Project Number: 14907.
\begin{center}
\includegraphics[scale=0.4]{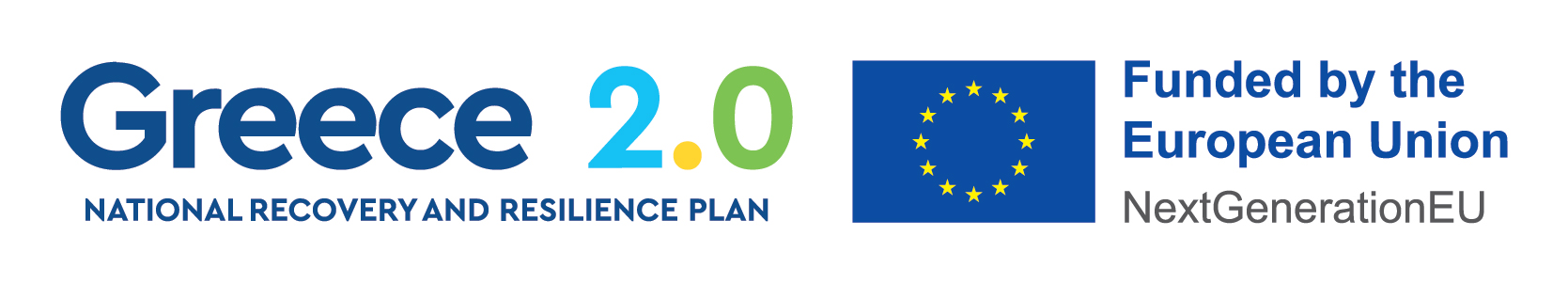}
\hskip 1cm
\includegraphics[scale=0.05]{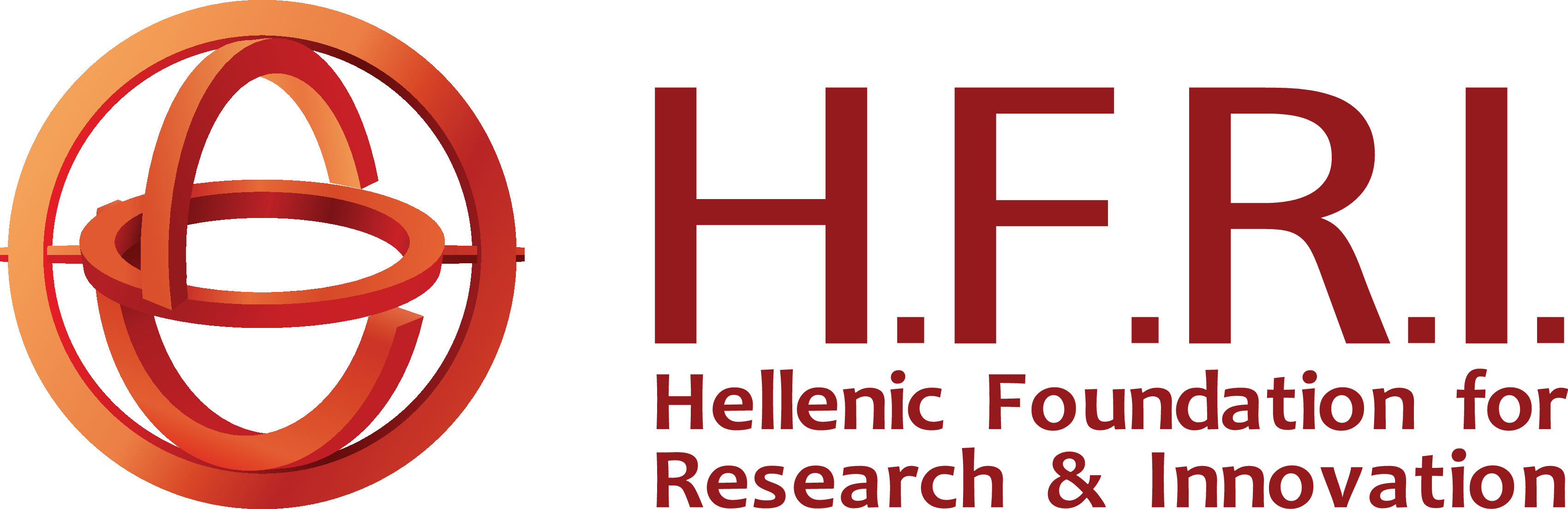}
\end{center}

\section{Deformation theory of curves with automorphisms}

\subsection{Global deformation functor}
\label{sec:DeformationTheory}

Let $\Lambda$ be a complete local Noetherian ring with residue field $k$, where $k$ is an algebraically closed field of characteristic $p\geq 0$. Let $\mathcal{C}$ be the category of local Artin $\Lambda$-algebras with residue field $k$ and homomorphisms the local $\Lambda$-algebra homomorphisms $\phi:\Gamma' \rightarrow \Gamma$, that is homomorphisms $\phi$ that satisfy  $\phi^{-1}(\mathfrak{m}_\Gamma)=\mathfrak{m}_{\Gamma'}$.
The deformation functor of curves with automorphisms 
is a functor $\Dgl$ from the category $\mathcal{C}$  to the category of sets 
\[
\Dgl: \mathcal{C} \rightarrow \rm{Sets}, 
\Gamma \mapsto
\left\{
\mbox{
\begin{tabular}{l}
Equivalence classes \\
of deformations of \\
couples $(X,G)$ over $\Gamma$
\end{tabular}
}
\right\}
\]
defined as follows. 
For a subgroup $G$ of the group  $\A(X)$, 
a deformation of the couple $(X,G)$ over the local Artin ring $\Gamma$ is a proper, smooth family
of curves 
\[
X_\Gamma \rightarrow  \Spe(\Gamma)
\]
parametrized by the base scheme $\Spe(\Gamma)$, together with a group homomorphism $G\rightarrow \A _\Gamma(X_\Gamma)$, such that there is a 
$G$-equivariant isomorphism $\phi$ 
from 
the fibre over the closed 
point of $\Gamma$ to the original curve $X$:
\[
\phi: X_\Gamma\otimes_{\Spe(\Gamma)} \Spe(k)\rightarrow X. 
\]
Two deformations $X_\Gamma^1,X_\Gamma^2$ are considered 
to be equivalent if there is a $G$-equivariant isomorphism $\psi$ that 
reduces to the identity in the special fibre and 
making the following diagram  commutative:
\[
\xymatrix@R-10pt{
X_\Gamma^1 \ar[rr]^{\psi} \ar[dr] & & X_\Gamma^2 \ar[dl] \\
& \Spe \Gamma &
}
\]
 Given a small extension of Artin local rings 
\begin{equation}
\label{smal-ext1}
0 \rightarrow E \cdot k \rightarrow \Gamma' \rightarrow \Gamma \rightarrow 0
\end{equation}
and an element  $x\in D_{\mathrm{gl}}(\Gamma)$, the set of lifts
 $x'\in D_{\mathrm{gl}}(\Gamma')$ extending $x$ is a principal homogeneous space under the action of $D_{\mathrm{gl}}(k[\epsilon])$ and such an extension $x'$ exists if a certain obstruction vanishes. 
It is well known, see section \ref{sec:unifyReptheories}, {that the deformation functors of representations have similar behavior}.

\subsection{Lifting of representations}
\label{sec:unifyReptheories}

Let $\mathcal{G}:\mathcal{C} \rightarrow \mathrm{Groups}$ be a group functor, see \cite[ch. 2]{MR883960}. In this article, we will be mainly interested in two 
group functors. The first one, $\GL_g$,   will be represented by the 
by the group scheme $G_g=\Lambda[x_{11},\ldots,x_{gg}, \det(x_{ij})^{-1}]$, that is $\GL_g(\Gamma)=\mathrm{Hom}_\Lambda(G_g,\Gamma)$. The second one is  the group functor from the category of rings to the category of groups $\mathcal{N}:\Gamma \mapsto  \Gamma[[t]]$.

We also assume that each group $\mathcal{G}(\Gamma)$ is embedded in the group of units of some ring $\mathcal{R}(\Gamma)$ depending functorially on $\Gamma$. This condition is asked since our  argument requires us to be able to add together certain group elements. 
We also assume that the additive group of the ring $\mathcal{R}(\Gamma)$ has the structure of direct product  $\Gamma^I$, while $\mathcal{R}(\Gamma)=\mathcal{R}(\Lambda)\otimes_\Lambda\Gamma$. Notice, that $I$ might be an infinite set, but since all rings involved are Noetherian $\Gamma^I$ is flat, see \cite[4F]{MR1653294}.

A representation of the finite  group $G$ in $\mathcal{G}(\Gamma)$ is a group homomorphism 
\[
\rho: G \rightarrow \mathcal{G}(\Gamma),
\]
where $\Gamma$ is a commutative ring. 

\begin{remark}
\label{action-operators}
Consider two sets  $X,Y$ acted on by the group $G$. Then, every function 
$f:X\rightarrow Y$ is acted on by $G$, by defining the function $^\sigma f:X\rightarrow Y$, sending $x \mapsto \sigma f \sigma^{-1}(x)$. This construction will be used throughout this article. 
\end{remark}

{ To avoid confusion, where necessary, we will denote with $\cdot$ for a group action and no symbol for the ring multiplication.}
More precisely  we will use the following actions
\begin{definition}\label{def:T-action}
\begin{enumerate}
    \item  Let $M_g(\Gamma)$ denote the set of $g\times g$ matrices with entries in ring $\Gamma$ {and a homomorphism of groups $\rho: G\rightarrow \mathrm{GL}_n(\Gamma)$}.
An element $A \in M_g(\Gamma)$ will be acted on by $\sigma \in G$ in terms of the action 

\[
T(\sigma) \cdot A=\rho(\sigma ^{-1})^t A \rho(\sigma ^{-1}). 
\]
This is the natural action coming from the action of $G$ on $H^0(X,\Omega_{X/k})$ and on the quadratic forms $\omega^t A \omega$. We raise the 
group element in $-1$ in order to have a left action, that is 
$T(\sigma _1 \sigma_2)A=T(\sigma _1)T(\sigma _2)A$. 
{Notice also that the action defined by $T$ restricts to an action} on the space $\mathcal{S}_{g}(\Gamma)$ of symmetric $g\times g$ matrices with entries in $\Gamma$. 
\item 
The adjoint action on elements $A \in M_g(\Gamma)$, comes from the action to the tangent space of the general linear group. 
\[
\mathrm{Ad}(\sigma ) \cdot A = \rho(\sigma ) A \rho(\sigma ^{-1}). 
\]
\item Actions on elements which can be seen as functions between $G$-spaces as in remark \ref{action-operators}. This action will be denoted as $f\mapsto ^\sigma\!\! f$.
\end{enumerate}
\end{definition}

\noindent {\bf Examples}

\noindent 
{\bf 1.} Consider the groups $\GL_g(\Gamma)$ consisted of all invertible $g\times g$ matrices with coefficients in $\Gamma$. 
{
Let $R$ be the affine $\Lambda$-algebra $R=k[x_{11},\ldots,x_{gg},\det\big( (x_{ij} ) \big)^{-1}]$.
}
The group functor 
\[
\Gamma \mapsto \GL_g(\Gamma)=\mathrm{Hom}(R,\Gamma), 
\] 
is representable { by $R$}, see \cite[2.5]{MR1638478}.
In this case the ring $\mathcal{R}(\Gamma)$ is equal to $\mathrm{End}(\Gamma^g)$, while $I=\{i,j \in \mathbb{N}: 1 \leq i,j, \leq g\}$.

We can consider the subfunctor $\GL_{g,\mathbb{Id}_g}$ consisted of all 
elements $f\in \GL_g(\Gamma)$, which reduce to the identity modulo the maximal ideal $\mathfrak{m}_\Gamma$. 
The tangent space 
$T_{\mathbb{I}_g}\GL_g$  of $\GL_g$ at the identity element $\mathbb{I}_g$, that is the space
 $\mathrm{Hom}(\Spe k[\epsilon],\Spe R)$ or equivalently the set
$\GL_{g,\mathbb{Id}_g}(k[\epsilon])$ consisted of 
$f\in \mathrm{Hom}(R,k[\epsilon])$, so that $f\equiv \mathbb{I}_g \mod \langle \epsilon \rangle$. This set is a vector space according to the functorial construction given in \cite[p. b 272]{MazDef} and 
can be identified to the space of $\mathrm{End}(k^g)=M_g(k)$, by identifying 
\[
\mathrm{Hom}(R,k[\epsilon]) \ni
 f \mapsto \mathbb{I}_g + \epsilon M, M \in M_g(k). 
\]
The later space is usually considered as the tangent space of the algebraic group $\GL_g(k)$ at the identity element or equivalently as the Lie algebra corresponding to $\GL_g(k)$.

The representation $\rho: G \rightarrow \GL_g(\Gamma)$ equips the space $T_{\mathbb{I}_g}\GL_g=M_g(k)$  with the adjoint action, which is the action described in remark \ref{action-operators}, when the endomorphism $M$ is seen as an operator $V\rightarrow V$, where $V$ is a $G$-module in terms of the representation $\rho$: 
\begin{align*}
G \times M_g(k) & \longrightarrow M_g(k) \\
{(\sigma,M)} & \longmapsto \mathrm{Ad}(\sigma) \cdot (M)=\rho(\sigma)M\rho(\sigma)^{-1}.
\end{align*}
 
In order to make clear the relation with the local case below, where the main object of study is the automorphism group of a completely local ring we might consider the completion $\hat{R}_{\mathbb{I}}$ of the localization of $R=k[x_{11},\ldots,x_{gg},\det\big( (x_{ij} ) \big)^{-1}]$ at the identity element. We can now form the group 
$\A \hat{R}_{\mathbb{I}}$ of automorphisms of the ring $\hat{R}_{\mathbb{I}}$  which reduce to the identity modulo $\mathfrak{m}_{\hat{R}_{\mathbb{I}}}$. The later automorphism group is huge, but it certainly contains the group $G$ acting on $\hat{R}_{\mathbb{I}}$
in terms of the adjoint representation. We have that 
{ 
an element $\sigma\in \A \hat{R}_{\mathbb{I}}\otimes k[\epsilon]$,  which is the image  of an element of $G$ in $\A \hat{R}_{\mathbb{I}}$} is of the form
 \[
 \sigma(x_{ij})=x_{ij}+ \epsilon \beta(x_{ij}), \text{ where }
 \beta(x_{ij}) \in \hat{R}_{\mathbb{I}}.
 \]
  Moreover, the relation 
 \[
\sigma(f g)= fg+ \epsilon \beta(f g) =
(f+\epsilon \beta(f))(g+ \epsilon \beta(f)),
 \]
 implies that the map $\beta$ is a derivation and 
 \[
\beta(f g)= f \beta(g)+ \beta(f)g.
 \]
Therefore, $\beta$ is a linear combination of $\frac{\partial}{\partial x_{ij}}$, with coefficients in $\hat{R}_{\mathbb{I}}$, that is 
\[
\beta=\sum_{0 \leqq i,j \leq g} a_{i,j} \frac{\partial}{\partial x_{ij}}.
\] 
\begin{remark}
In the literature of Lie groups and algebras,  the matrix notation  $M_g(k)$ for the tangent space is frequently used for the Lie algebra-tangent space at identity,  instead of the later vector field-differential operator approach, while in the next example, the differential operator notation for the tangent space is usually used. 
\end{remark}


\noindent
{\bf 2.} Consider now the group functor $\Gamma \mapsto \mathcal{N}(\Gamma)=\A \Gamma[[t]]$. An element   $\sigma\in \A \Gamma[[t]]$ is fully described by its action on $t$, which can be expressed as an element in $\Gamma[[t]]$. When $\Gamma$ is an Artin local algebra then an automorphism is given by  
\[
\sigma(t)= \sum_{\nu=0}^\infty a_\nu t^\nu, \text{ where } a_i\in \Gamma, a_0 \in \mathfrak{m}_\Gamma \text{ and } a_1 \text{ is a unit in } \Gamma.
\] 
If $a_1$ is not a unit in $\Gamma$ or $a_0 \not \in \mathfrak{m}_\Gamma$, then $\sigma$ is an endomorphism of 
$\Gamma[[t]]$. 
In this way  $\A \Gamma[[t]]$ can be seen as the group of invertible elements in $\Gamma[[t]]=\mathrm{End} \Gamma[[t]]=\mathcal{R}(\Gamma)$.
The set $I$ is equal to the set of natural numbers, where $\Gamma^I$ can be identified as the set of coefficients of each powerseries.

\begin{align*}
\A (k[\epsilon][[t]])&=
\left\{
t\mapsto \sigma(t) = { \epsilon\beta_0 +} \sum_{\nu=1}^\infty a_i t^\nu: a_i=\alpha_i+ \epsilon \beta_i, \ \alpha_i,\beta_i\in k, \alpha_1\neq 0
\right\}.
\end{align*}
Exactly as we did in the general linear group case,
let us consider the subfunctor $\Gamma \mapsto \mathcal{N}_{\mathbb{I}}(\Gamma)$, where $\mathcal{N}_{\mathbb{I}}(\Gamma)$ consists of all elements in $\A \Gamma[[t]]$, which reduce to the identity mod $\mathfrak{m}_\Gamma$.

Such an element  $\sigma \in \mathcal{N}_{\mathbb{I}}(k[\epsilon])$ transforms $f\in k[[t]]$ to a formal powerseries of the form 
\[
\sigma(f)=f + \epsilon F_\sigma(f),  
\]
where $F_\sigma(f)$ is fully determined by the value of $\sigma(t)$. 
The multiplication condition $\sigma(f_1 f_2)=\sigma(f_1)\sigma(f_2)$ implies that 
\[
F_\sigma(f_1 f_2)=f_1 F_\sigma(f_2) + F_\sigma(f_1) f_2, 
\]
that is $F_\sigma$ is a $k[[t]]$-derivation, hence an element in $k[[t]]\frac{d}{dt}$.

The local tangent space of $\Gamma[[t]]$ is defined to be the space of differential operators $f(t)\frac{d}{dt}$, see \cite{Be-Me}, \cite{CK}, \cite{KontoANT}. 
The $G$ action on the element $\frac{d}{dt}$ is given by the adjoint action, which is given as a composition of operators, and is again compatible with the action given in remark \ref{action-operators}:

\[
\xymatrix@R-17pt{
    \Gamma[[t]]  \ar[r]^{\rho(\sigma^{-1})} &
     \Gamma[[t]]  \ar[r]^{\frac{d}{dt}}  &
    \Gamma[[t]]  \ar[r]^{\rho(\sigma)} &
    \Gamma[[t]]
    \\
t \ar@{|->}[r]&
 \rho(\sigma^{-1})(t) 
  \ar@{|->}[r]&
  \frac{d\rho(\sigma^{-1})(t)}{dt}
  \ar@{|->}[r]
  &
  \rho(\sigma)
  \left(
  \frac{d\rho(\sigma^{-1})(t)}{dt} 
  \right)
}
\]   
So the $G$-action on the local tangent space $k[[t]]\frac{d}{dt}$ is given by 
\[
f(t)\frac{d}{dt}\longmapsto 
\mathrm{Ad}(\sigma)
\left( 
f(t)\frac{d}{dt}
\right)
=
\rho(\sigma)(f(t)) \cdot  \rho(\sigma)
  \left(
  \frac{d\rho(\sigma^{-1})(t)}{dt} 
  \right) \frac{d}{dt}, 
\]
see also \cite[lemma 1.10]{KontoANT}, for a special case.

\begin{table}[h]
\centerline{
\newcolumntype{g}{>{\columncolor{ashgrey!15}}c}
\newcolumntype{a}{>{\columncolor{ashgrey!40}}c}
\begin{tabular}{|agag|} 
\toprule
\rowcolor{blue!15}
$\mathcal{G}(\Gamma)$ & $\mathcal{R}(\Gamma)$  & tangent space & action \\
\midrule
 $\GL_g(\Gamma)$ & $\mathrm{End}_g(\Gamma)$ & $\mathrm{End}_g(k)=M_g(k)$  & $ M \mapsto \mathrm{Ad}(\sigma)(M)$ \\
$\A \Gamma[[t]]$ & $\mathrm{End}(\Gamma[[t]])$ & $k[[t]]\frac{d}{dt}$ 
&
$f(t)\frac{d}{dt}\longmapsto 
\mathrm{Ad}(\sigma)
\left( 
f(t)\frac{d}{dt}
\right)$
 \\
\bottomrule
\end{tabular}
}
\caption{Comparing the two group functors}
\end{table}

Motivated by the above two examples we can define 
\begin{definition}
Let $\mathcal{G}_{\mathbb{I}}$ be the subfunctor of $\mathcal{G}$, defined by  
\[
\mathcal{G}_{\mathbb{I}}(\Gamma)=\{f\in \mathcal{G}(\Gamma): f =\mathbb{I} \mod \mathfrak{m}_\Gamma\}.
\]
The tangent space to the functor $\mathcal{G}$ at the identity element is defined as 
$
\mathcal{G}_{\mathbb{I}}(k[\epsilon])
$, see \cite{MazDef}.
Notice, that $\mathcal{G}_{\mathbb{I}}(k[\epsilon])\cong \mathcal{R}(k)$, is $k$-vector space, acted on in terms of the adjoint representation, given by 
\begin{align*}
G \times \mathcal{G}_{\mathbb{I}}(\Gamma) & \longrightarrow \mathcal{G}_{\mathbb{I}}(\Gamma) \\
(\sigma,f) &\longmapsto \rho(\sigma) f \rho(\sigma)^{-1}.
\end{align*} 
If $\mathcal{R}(\Gamma)$ can be interpreted as an endomorphism ring, then the above action can be interpreted in terms of the action on functions as described in remark \ref{action-operators}. 

We will define the tangent space in our setting as   $\mathcal{T}=\mathcal{R}(k)$, which is equipped with the adjoint action. 
\end{definition}
%

\subsection{Deforming representations}

We can now define the deformation functor $F_\rho$ for any local Artin algebra $\Gamma$ with maximal ideal  $\mathfrak{m}_\Gamma$ in $\mathcal{C}$  to the category of sets: 
\begin{equation} \label{Fdeformation}
F_\rho: \Gamma \in \mathrm{Ob}(\mathcal{C}) \mapsto \left\{
\begin{array}{l}
\mbox{liftings of } \rho: G \rightarrow \mathcal{G}(k) \\
\mbox{to } \rho_\Gamma: G \rightarrow \mathcal{G}(\Gamma) 
\mbox{ modulo} \\ \mbox{conjugation by an element }\\
\mbox{of }  
\ker(\mathcal{G}(\Gamma)\rightarrow \mathcal{G}(k))
\end{array}
\right\}
\end{equation}
Let 
\begin{equation}
\label{small-ex}
\xymatrix{
    0 \ar[r] &
     \langle E \rangle=E \cdot \Gamma' =E \cdot k 
\ar[r]_-{\phi'} &
\Gamma'  \ar[r]_\phi 
&
\Gamma \ar[r]  
\ar@/_1.0pc/@[red][l]_-i
& 0 
}
\end{equation}
be a small extension in $\mathcal{C}$, that is the kernel of the natural onto map $\phi$ is a principal ideal, generated by $E$ and $E \cdot \mathfrak{m}_{\Gamma'}=0$.
In the above diagram   $i:\Gamma \rightarrow \Gamma'$ is a section, which is not necessarily a homomorphism. Since  the kernel of $\phi$ is a principal ideal $E \cdot \Gamma'$ annihilated by $\mathfrak{m}_{\Gamma'}$ it is naturally a $k=\Gamma'/\mathfrak{m}_{\Gamma'}$-vector space, which is one dimensional.

\begin{lemma}
\label{coc-def1}
For a small extension as given in eq. (\ref{small-ex}) consider two liftings $\rho^1_{\Gamma'}, \rho^2_{\Gamma' }$ of the representation
 $\rho_\Gamma$.
The map  
\begin{align*}
d:G &\longrightarrow \mathcal{T}:=\mathcal{R}(k) \\
\sigma & \longmapsto d(\sigma)=
\frac{
 \rho^1_{\Gamma'}(\sigma) 
 \rho^2_{\Gamma'} (\sigma)^{-1}
 - \mathbb{I}_{\Gamma'}
 }
 {
 E
 } 
\end{align*}
is a cocycle. 
\end{lemma}
\begin{proof}
We begin by observing that 
$
\phi  
\left(
\rho^1_{\Gamma'}(\sigma) 
 \rho^2_{\Gamma'}(\sigma)^{-1} - \mathbb{I}_{\Gamma'}
\right)
=0,
$
hence 
\[
\rho^1_{\Gamma'}(\sigma) \rho^2_{\Gamma'}(\sigma)^{-1} =\mathbb{I}_{\Gamma'}
+ E\cdot  d(\sigma), \text{ where } d(\sigma) \in \mathcal{T}. 
\]
Also, we compute that 
{
\begin{align*}
\mathbb{I}_{\Gamma'}+ 
E \cdot 
d(\sigma \tau) &= 
 \rho^1_{\Gamma'}(\sigma \tau) 
  \rho^2_{\Gamma'}(\sigma \tau)^{-1}
\\
&=
 \rho^1_{\Gamma'}(\sigma) \rho^1_{\Gamma'}( \tau)
  \rho^2_{\Gamma'}(\tau) ^{-1} 
   \rho^2_{\Gamma'}( \sigma)^{-1}
\\
&=
 \rho^1_{\Gamma'}( \sigma)
\big(
\mathbb{I}_{\Gamma'}+ E \cdot d(\tau)
\big)
  \rho^2_{\Gamma'}(\sigma)^{-1}
 \\
 &=
  \rho^1_{\Gamma'}(\sigma) 
  \rho^2_{\Gamma'}(\sigma)^{-1}
 +
E \cdot 
\rho^1_{\Gamma'}( \sigma)
d(\tau)
 \rho^2_{\Gamma'}(\sigma)^{-1} \\
 &=
 \mathbb{I}_{\Gamma'} 
 + E \cdot d(\sigma)
 + E  \cdot 
 \rho_k(\sigma) d(\tau) \rho_k(\sigma)^{-1}, 
\end{align*}}
since $E$ annihilates $\mathfrak{m}_{\Gamma'}$, 
so the values of both $\rho^1_{\Gamma'}(\tau)$ and $\rho^2_{\Gamma'}(\tau)$ when multiplied by $E$ are reduced modulo the maximal ideal $\mathfrak{m}_{\Gamma'}$. 
Therefore, we  conclude that 
\[
d(\sigma \tau)=d (\sigma)+ \rho_k(\sigma) d(\tau) \rho_k(\sigma)^{-1}=
d (\sigma)+ \mathrm{Ad}(\sigma) \cdot d(\tau). 
\]
\end{proof}
Similarly if $\rho^1_{\Gamma'}, \rho^2_{\Gamma'}$ are equivalent extensions of $\rho_\Gamma$, that is 
\[
\rho^1_{\Gamma'}(\sigma)= 
\big( 
\mathbb{I}_{\Gamma'}+ E Q
\big)
\rho^2_{\Gamma'}(\sigma) 
\big(
\mathbb{I}_{\Gamma'}+ E Q
\big)^{-1},
\]
then 
\[
d(\sigma)= Q - \mathrm{Ad}(\sigma)Q,
\]
that is $d(\sigma)$ is a coboundary. 
This proves that the set of liftings $\rho_{\Gamma'}$ of a representation 
$\rho_{\Gamma'}$ is a principal homogeneous space, provided it is non-empty. 

The obstruction to the lifting can be computed by considering a naive lift $\rho_{\Gamma'}$ of $\rho_{\Gamma}$ (that is we don't assume that $\rho_{\Gamma'}$ is a representation) and by considering the element 
\[
\phi(\sigma,\tau)= \rho_{\Gamma'}(\sigma) \circ \rho_{\Gamma'}(\tau) \circ \rho_{\Gamma'}(\sigma \tau)^{-1}, 
 \quad \text{ for } \sigma,\tau\in G
\]
which defines a cohomology class  as an element in
 $H^2(G,\mathcal{T})$. Two naive liftings $\rho^1_{\Gamma'}, \rho^2_{\Gamma'}$ give rise to cohomologous elements $\phi^{1}, \phi^{2}$ if their difference $\rho^1_{\Gamma'}-\rho^2_{\Gamma'}$ reduce to zero in $\Gamma'$.
 If this class is zero, then the representation $\rho_{\Gamma}$ can be lifted to $\Gamma'$.  

\noindent
{\bf Examples}
Notice that in the theory of deformations of representations of the general linear group, this is a classical result, see \cite[prop. 1]{MazDef}, \cite[p.30]{MR818915} while for deformations of representations in $\A \Gamma [[t]]$, this is in \cite{CK},\cite{Be-Me}. 

The functors in these cases are given by 
\begin{equation} \label{Fdeformation1}
F: \mathrm{Ob}(\mathcal{C}) \ni \Gamma  \mapsto \left\{
\begin{array}{l}
\mbox{liftings of } \rho: G \rightarrow \GL_n(k) \\
\mbox{to } \rho_\Gamma: G \rightarrow \GL_n(\Gamma) 
\mbox{ modulo} \\ \mbox{conjugation by an element }\\
\mbox{of }  
\ker(\GL_n(\Gamma)\rightarrow \GL_n(k))
\end{array}
\right\}
\end{equation}

\begin{equation} \label{Bertin-Mezard-functor1}
D_P:
\mathrm{Ob}(\mathcal{C}) 
\ni \Gamma \mapsto 
\left\{
\mbox{
{
\begin{tabular}{l}
 liftings of $\rho: G\rightarrow \A k[[t]]$ \\
to $\rho_{\Gamma}: G\rightarrow \A \Gamma[[t]]$ modulo \\
conjugation by an element \\ of $\ker\left(\A \Gamma[[t]]\rightarrow \A k[[t]] \right)$
\end{tabular}
}}
\right\}
\end{equation} 

Let $V$ be the $n$-dimensional {$k$-vector space equipped with an action of $G$ given by the representation $\rho: G \rightarrow \mathrm{GL}(V)$}, and let $\mathrm{End}_A(V)$ be the  Lie algebra corresponding to the algebraic group $\mathrm{GL}(V)$. 
The space $\mathrm{End}_A(V)$ is equipped with the adjoint action of $G$ given by:
\begin{align*}
\mathrm{End}_A(V) & \rightarrow \mathrm{End}_A(V) 
\\
e & \mapsto 
(g\cdot e)(v)=\rho(g) (e (\rho(g)^{-1})(v) )
\end{align*}
The tangent space of this deformation functor equals to
\[
F(k[\epsilon])=H^1(G,\mathrm{End}_A(V)), 
\]
where the later cohomology group is the group cohomology group and 
$\mathrm{End}_A(V)$ is considered as a $G$-module with the adjoint action.

More precisely, if 
\[
0 \rightarrow \langle E \rangle 
\rightarrow 
\Gamma'
\stackrel{\phi}{\longrightarrow}
\Gamma
\rightarrow 
0
\] 
is a small extension of local Artin algebras then we consider the diagram of small extensions
\[
\xymatrix{
 &  \GL_n(\Gamma') \ar[d]^{\phi}
 \\
G  \ar[r]_-{\rho_\Gamma} \ar[ru]^{\rho^1_{\Gamma'},\rho^2_{\Gamma'}}
    & 
    \GL_n(\Gamma)
}
\]
where $\rho^1_{\Gamma'},\rho^2_{\Gamma'}$ are two liftings of $\rho_\Gamma$ in $\Gamma'$.

We have the element 
\[
d(\sigma):=
\frac{1}{E}\left( \rho^1_{\Gamma'}(\sigma)\rho^2_{\Gamma'}(\sigma)^{-1}-
\mathbb{I}_n
\right)
\in H^1(G,\mathrm{End}_n(k)).
\]
To a naive lift $\rho_{\Gamma'}$ of $\rho_{\Gamma}$ we can attach the 2-cocycle
$\alpha(\sigma,\tau)=\rho_{\Gamma'}(\sigma) \rho_{\Gamma'}(\tau)\rho_{\Gamma'}(\sigma \tau)^{-1}$, defining a cohomology class in $H^2(G,{\mathrm{End}_k(V)})$.

The following proposition shows us that a lifting is not always possible. 
\begin{proposition}
\label{prop:counter-lift}
Let $k$ be an algebraically closed field of positive characteristic $p>0$, end let $R=W(k)[\zeta_q]$ be the Witt ring of $k$ with a primitive $q=p^h$ root adjoined. 
Consider the group $G=C_q \rtimes C_m$, where $C_m$ and $C_q$ are cyclic groups of orders $m$ and $q$ respectively and $(m,p)=1$.  
Assume that $\sigma$ and $\tau$ are generators for $C_m$ and $C_q$ respectively and moreover
\[
\sigma \tau \sigma^{-1}= \tau^a
\]
for some integer $a$ (which should satisfy $a^m\equiv 1 \mod q$).
{ There are selections of $m, q$ such that the linear }
 representation $\rho:G\rightarrow \GL_2(k)$ can not be lifted to a representation $\rho_R:G \rightarrow \GL_2(R)$. 
\end{proposition}
\begin{proof}
We will construct only a faithful representation of $C_p \rtimes C_{p-1}$ in 
$\GL_2(k)$.  
Consider the field $\mathbb{F}_p \subset k$ and let $\lambda$ be a generator of the cyclic group $\mathbb{F}_p^*$. 
The matrices 
\[
\sigma=\begin{pmatrix}
\lambda  & 0 \\ 0 & 1
\end{pmatrix}
\text{ and }
\tau=
\begin{pmatrix}
1 & 1 \\ 0 & 1
\end{pmatrix}
\]
satisfy 
\[
\sigma^{p-1}=1, \tau^q=1,
\sigma \tau \sigma^{-1}= 
\begin{pmatrix}
1 & \lambda  \\
0 & 1
\end{pmatrix}=\sigma^\lambda 
\]
and generate 
a subgroup of $\GL_2(k)$, isomorphic to $C_p \rtimes C_m$ for $m=p-1$, giving a natural representation $\rho:G \rightarrow \GL_2(\bar{\mathbb{F}}_p) \subset \GL_2(k)$. 

Suppose that there is a faithful representation $\tilde{\rho}: G \rightarrow \GL_n(R)$, which gives a  faithful representation of $\tilde{\rho}: G\rightarrow  \GL_n(\mathrm{Quot}(R))$. 
Since $\tilde{\rho}(\tau)$ is of finite order, after a $\mathrm{Quot}(R)$ linear change of basis we might assume that $\tilde{\rho}(\tau)$ is diagonal with $q$-roots of unity in the diagonal (we have considered $R=W(k)[\zeta]$ so that the necessary diagonal elements exist in $\mathrm{Quot}(R)$). We have 
\[
\tilde{\rho}(\tau)=\mathrm{diag}(\lambda_1,\ldots,\lambda_n). 
\]
At least one of the diagonal elements say $\lambda=\lambda_{i_0}$ in the above expression is a primitive $q$-th root of unity.  
Let $E$ be an eigenvector, that is 
\[
\tilde{\rho}(\tau) E= \lambda E. 
\]
The equality $ \tau \sigma = \sigma \tau^a$ implies that 
$\sigma E$ is an eigenvector of the eigenvalue $\lambda^a$. This means that $n$ should be greater than the order of $a \mod q$ since we have  at least as many different (and linearly independent) eigenvectors as the different values $\lambda, \lambda^a, \lambda^{a^2},\ldots$. 

Since, for large prime ($p>3$) we have $2=n < p-1$ the representation $\rho$ can not be lifted to $R$. 
\end{proof}
\begin{remark}
In \cite{MR4779377} we give a necessary and sufficient condition for a modular representation of a group $C_{p^h} \rtimes C_m$ in a field of characteristic $p>0$ to be lifted to a representation over a local principal ideal domain of characteristic zero containing the $p^h$ roots of unity.
\end{remark}

{\bf Local Actions}
By the local-global theorems of J.Bertin and A. M\'ezard  \cite{Be-Me} and the formal patching theorems of 
D. Harbater,  K. Stevenson \cite{HarMSRI03}, \cite{HarStevJA99},  the study of the functor $\Dgl$ can be reduced to the study of the deformation 
functors  $D_P$ attached to each wild ramification point  $P$ of the cover $X \rightarrow X/G$, as defined in eq. (\ref{Bertin-Mezard-functor1}).
The theory  of automorphisms of formal powerseries rings is not as well understood as  is 
the theory of automorphisms of finite dimensional vector spaces, i.e. the theory of general linear groups.

As in the theory of liftings for the general linear group, we consider small extensions
\[
1 \rightarrow \langle E \rangle 
\rightarrow
\Gamma' 
\stackrel{\phi}{\longrightarrow}
\Gamma
\rightarrow 
1.
\]
{ Let $\sigma$ be an element in the (finite) group $G$.}
An automorphism $\rho^{\Gamma}(\sigma) \in  { \A \Gamma[[t]]}$, { corresponding to $\sigma$,} is completely 
described by a powerseries 
\[
\rho^{\Gamma}(\sigma)(t)=f_\sigma=\sum_{{\nu=0}}^\infty a_\nu^{\Gamma}(\sigma) t^\nu,
\] 
where $a_\nu^{\Gamma}(\sigma) \in \Gamma$.
Given a naive lift 
\[
\rho^{\Gamma'}(\sigma)(t)=\sum_{{\nu=0}}^\infty a_\nu^{\Gamma'}(\sigma) t^\nu,
\]
 where $a_\nu^{\Gamma'}(\sigma) \in \Gamma'$ we can again form a 2-cocycle
\[
\alpha(\sigma,\tau)=\rho^{\Gamma'}(\sigma) \circ \rho^{\Gamma'}(\tau)
\circ
\rho^{\Gamma'}(\sigma \tau)^{-1}(t),
\]
defining a cohomology class in $H^2(G,\mathcal{T}_{k[[t]]})$. The naive lift $\rho^{\Gamma'}(\sigma)$ is an element of ${ \A \Gamma'[[t]]}$  if and only if $\alpha$ is cohomologous to zero. 

Suppose now that $\rho_1^{\Gamma'}, \rho_2^{\Gamma'}$ are two lifts in ${ \A \Gamma'[[t]]}$. 
We can now define 
\[
d(\sigma):=
\frac{1}{t}\left( \rho_1^{\Gamma'}(\sigma)\rho_2^{\Gamma'}(\sigma)^{-1}-
\mathrm{Id}
\right)
\in H^1(G,\mathcal{T}_{k[[t]]}).
\]

\section{Relative Petri's theorem.}
\label{sec:RelPetriThm}

Recall that a functor $F: \mathcal{C} \rightarrow \mathrm{Sets}$ can be extended to a functor  $\hat{F}: \hat{\mathcal{C}} \rightarrow \mathrm{Sets}$ by letting
$\displaystyle \hat{F}(R)=\lim_{\leftarrow} F(R/\mathfrak{m}_R^{n+1})$
for every $R\in \mathrm{Ob}(\hat{\mathcal{C}})$. An element $\hat{u} \in \hat{F}(R)$ is called a formal element, and by definition it can be represented as a system of elements $\{u_n \in F(R/\mathfrak{m}_R^{n+1})\}_{n\geq 0}$, such that for each $n\geq 1$, the map $F(R/\mathfrak{m}_R^{n+1}) \rightarrow F(R/\mathfrak{m}_{R}^n)$ induced by $R/\mathfrak{m}_R^{n+1} \rightarrow R/\mathfrak{m}_R^n$ sends $u_n\mapsto u_{n-1}$. For $R \in \mathrm{Ob}(\hat{\mathcal{C}})$ and a formal element  $\hat{u}\in \hat{F}(R)$, the couple $(R,\hat{u})$ is called a formal couple. It is known that there is a 1-1 correspondence between $\hat{F}(R)$ and  the set of morphisms of functors $h_R:=\mathrm{Hom}_{\hat{\mathcal{C}}}(R,-) \rightarrow F$, see \cite[lemma 2.2.2.]{MR2247603}. The formal element $\hat{u}\in \hat{F}(R)$ will be called versal if the corresponding morphism $h_R\rightarrow F$ is smooth. For the definition of a smooth map between functors, see \cite[def. 2.2.4]{MR2247603}. The ring $R$ will be called {\em versal deformation ring}.

M. Schlessinger in \cite[3.7]{Sch} proved that the deformation functor $D$ for curves without automorphisms,  admits a ring $R
$ as versal deformation ring. 
Schlessinger calls the versal deformation ring {\em the hull of the deformation functor}. 
Indeed, since there are no obstructions to liftings in small extensions for curves, see \cite[rem. 2.10]{Sch} the hull $\Rgl$ of $\Dgl$ is a powerseries ring over $\Lambda$, which can be taken as the ring of integers in an algebraic extension of the fraction field of $W(k)$. Moreover 
$
\Rgl=\Lambda[[x_1,\ldots,x_{3g-3}]],
$
as we can see by applying  \cite[cor. 3.3.5]{BeMe2002}, when $G$ is the trivial subgroup of the automorphism group. In this case the quotient  map $f:X \rightarrow \Sigma=X/\{\mathrm{Id}\}=X$ is the identity. Indeed, 
 for the equivariant deformation functor, in the case of 
the trivial group, there are no ramified points and the short exact sequence in eq. (\ref{BeME-lg}) reduces to an isomorphism of the first two spaces. 
We have $\dim_k H^1(X/G, \pi_*^G (\mathcal{T}_X))=\dim_k H^1(X,\mathcal{T}_X)=3g-3$.
The deformation $\mathcal{X} \rightarrow {\mathrm{Spf}}\Rgl$ can be extended to a deformation $\mathcal{X} \rightarrow \mathrm{Spec} \Rgl$ by Grothendieck's effectivity theorem, see \cite[th. 2.5.13]{MR2247603}, \cite{MR1603467}.

The versal element $\hat{u}$ corresponds to a deformation  $\mathcal{X}\rightarrow \mathrm{Spec}R$,  with generic fibre $\mathcal{X}_\eta$ and special fibre $\mathcal{X}_0$. 
The couple $(R,\hat{u})$ is  called the versal
\cite[def. 2.2.6]{MR2247603} element
of the deformation functor $D$ of curves (without automorphisms). 
Moreover, the element $u$ defines a   map
$h_{R/\Lambda}\rightarrow D$, which by definition of the hull is smooth, so every deformation
$X_A \rightarrow \Spe A$ gives rise to a  non-canonical homomorphism $R \rightarrow A$, which allows us to see $A$ as an $R$-algebra.  
 Indeed, for the Artin algebra $A \rightarrow A/\mathfrak{m}_A=k$ we consider the diagram
\[
\xymatrix{
	h_{R/\Lambda}=\mathrm{Hom}_{\widehat{\mathcal{C}}}(R,A)
	\rightarrow 
	h_{R/\Lambda}(k) \times_{D(k)} D(A).
	}
\]

This section aims to prove proposition \ref{red-new-prot}.
For $n\geq 1$, we write $\Omega_{\mathcal{X}/R}^{\otimes n}$ for the sheaf of holomorphic polydifferentials on $\mathcal{X}$. By \cite[lemma II.8.9]{Hartshorne:77} the $R-$modules $H^0(\mathcal{X},\Omega^{\otimes n}_{\mathcal{X}/R})$ are free of rank $d_{n,g}$ for all $n\geq 1$, with $d_{n,g}$ given by eq. (\ref{dng})
\begin{equation}\label{dng}
d_{n,g}=
\begin{cases}
g, & \text{ if } n=1\\
(2n-1)(g-1), & \text{ if } n>1.
\end{cases} 
\end{equation}
Indeed, by a standard argument using Nakayama's lemma, see \cite[lemma II.8.9]{Hartshorne:77},\cite{KaranProc} we have that the $\Rgl$-module  $H^0(\mathcal{X},\Omega^{\otimes n}_{\mathcal{X}/\Rgl})$ is free.  
Notice that in order to use Nakayama's lemma we need the deformation over  $\Rgl$ to have both a special and generic fibre and this was the reason we needed to consider a deformation over the spectrum of $\Rgl$ instead of the formal spectrum.

\begin{lemma}
For every Artin algebra $A$ the $A$-module $H^0(X_A,\Omega_{X_A/A}^{\otimes n})$ is free.
\end{lemma}
\begin{proof}
This follows since $H^0(\mathcal{X},\Omega_{\mathcal{X}/\Rgl})$
is a free $\Rgl$-module and {by} \cite[prop. II.8.10]{Hartshorne:77}, 
which asserts that $\Omega_{X_A/A} \cong g^{\prime *} (\Omega_{\mathcal{X}/\Rgl})$, where $g'$ is shown in the next commutative diagram:
\[
\xymatrix{
	X_A=\mathcal{X} \times_{\Spe \Rgl} \Spe A \ar[r]^-  {g'} \ar[d] &
	\mathcal{X} \ar[d] \\
	\Spe A \ar[r] & \Spe \Rgl
}
\]
We have by definition of the pullback 
\begin{equation}
\label{pullback}
g^{\prime *} (\Omega_{\mathcal{X}/\Rgl})(X_A) = (g')^{-1}\Omega_{\mathcal{X}/R}(X_A)\otimes_{(g')^{-1}\mathcal{O}_\mathcal{X}(X_A)}\mathcal{O}_{X_A}(X_A)
\end{equation}
and by definition of the fibre product we obtain
$\mathcal{O}_{X_A}=\mathcal{O}_{\mathcal{X}}\otimes_{R}A$. 
Observe also that since $A$ is a local Artin algebra the schemes $X_A$ and $\mathcal{X}$ share the same underlying topological space
so 
\[
g^{\prime -1}(\Omega_{\mathcal{X}/\Rgl}(X_A))=
 \Omega_{\mathcal{X}/\Rgl}(\mathcal{X})\] and 
$g^{\prime -1} \mathcal{O}_{\mathcal{X}}(X_A)=\mathcal{O}_{\mathcal{X}}(\mathcal{X})$. So eq. (\ref{pullback}) becomes
\begin{align*}
H^0(X_A,\Omega_{X_A/A}) & = \Omega_{X_A/A}(X_A)=g^{\prime *}(\Omega_{\mathcal{X}/\Rgl})(X_A))=
\\
&= 
\Omega_{\mathcal{X}/\Rgl}(\mathcal{X}) 
\otimes_{\mathcal{O}_{\mathcal{X}}(\mathcal{X})} \otimes 
{\mathcal{O}_{\mathcal{X}}}(\mathcal{X})  \otimes_{\Rgl}A
\\
&= H^0(\mathcal{X},\Omega_{\mathcal{X}/\Rgl}) \otimes_{\Rgl} A.
\end{align*}
So $H^0(X_A,\Omega_{X_A/A})$ is a free $A$-module of the same rank as $H^0(\mathcal{X},\Omega_{\mathcal{X}/\Rgl})$. 

The proof for $H^0(X_A, \Omega_{X_A/A}^{\otimes n})$ follows in the same way. 
\end{proof}

 We select generators $W_1,\ldots,W_g$ for the symmetric algebra
\[
 \Sym (H^0(\mathcal{X},\Omega_{\mathcal{X}/R}))=R[W_1,\ldots,W_g].
 \]
  Similarly, for $L= \mathrm{Quot}(R)$  we write 
  \[ 
  \Sym (H^0(\mathcal{X}_\eta,\Omega_{\mathcal{X}_\eta/L}))=L[\omega_1,\ldots,\omega_g] 
  \text{ and  }
  \Sym (H^0(\mathcal{X}_0,\Omega_{\mathcal{X}_0/k}))=k[w_1,\ldots,w_g],
  \] 
 where 
 \[
\omega_i=W_i \otimes_R L \qquad w_i=W_i \otimes_Rk \text{ for all } 1\leq i \leq g.
 \]
 We have the following diagram relating special and generic fibres
\begin{equation}\label{classic-diagram}
\begin{tikzcd}
\mathrm{Spec}(k)\times_{\mathrm{Spec}(R)}\mathcal{X}=\mathcal{X}_0\arrow[hookrightarrow]{r}{}\arrow[rightarrow]{d}{}& 
\mathcal{X}\arrow[hookleftarrow]{r}{} \arrow[rightarrow]{d}&\mathcal{X}_\eta=\mathrm{Spec}(L)\times_{\mathrm{Spec}(R)}\mathcal{X} \arrow[rightarrow]{d}{}\\  
\mathrm{Spec}(k)\arrow[hookrightarrow]{r}{} &\mathrm{Spec}(R)\arrow[hookleftarrow]{r}{}    &\mathrm{Spec}(L)
\end{tikzcd}
\end{equation}

Our article is based on the following relative version of Petri's theorem
\begin{theorem}\label{relative-canonical-embedding}
{ Let $\mathcal{X} \rightarrow \Spe R$ be a relative curve, such that the special fibre $\mathcal{X}_0$ satisfies the assumptions of Petri's theorem and its canonical ideal $I_{\mathcal{X}_0}$ is generated by quadratic polynomials.}
Diagram (\ref{classic-diagram}) induces a deformation-theoretic diagram of canonical embeddings
\begin{equation} \label{gen-diagram}
\xymatrix{
	0 \ar[r] & I_{\mathcal{X}_\eta}\ar@{^{(}->}[r] & S_L:=L[\omega_1,\ldots,\omega_g] \ar@{->>}[r]^-{\phi_\eta} & 
	\displaystyle\bigoplus_{n=0}^\infty H^0(\mathcal{X}_\eta,\Omega_{\mathcal{X}_\eta/L}^{\otimes n}) \ar[r] & 0  
	\\
	0 \ar[r] &
	 I_{\mathcal{X}}\ar@{^{(}->}[r] \ar@{^{(}->}[u]_{\otimes_R L}   \ar@{->>}[d]^{\otimes_R R/\mathfrak{m}}
	 & 
	S_R:=R[W_1,\ldots,W_g] \ar@{->>}[r]^-{\phi} \ar@{^{(}->}[u]_{\otimes_R L}  \ar@{->>}[d]^{\otimes_R R/\mathfrak{m}}
	& 
	\displaystyle\bigoplus_{n=0}^\infty
	H^0(\mathcal{X},\Omega_{\mathcal{X}/R}^{\otimes n}) \ar[r] \ar@{^{(}->}[u]_{\otimes_R L}  \ar@{->>}[d]^{\otimes_R R/\mathfrak{m}}
	& 0 
	\\
	0 \ar[r] & I_{\mathcal{X}_0}\ar@{^{(}->}[r] & S_k:=k[w_1,\ldots,w_g] \ar@{->>}[r]^-{\phi_0} &
	\displaystyle\bigoplus_{n=0}^\infty H^0(\mathcal{X}_0,\Omega_{\mathcal{X}_0/k}^{\otimes n}) \ar[r] & 0 
}
\end{equation}
where $I_\mathcal{X_\eta}=\ker\phi_\eta,\;I_{\mathcal{X}}=\ker\phi,\;I_{\mathcal{X}_0}=\ker\phi_0$, each row is exact and each square is commutative. Moreover, the ideal $I_\mathcal{X}$ can be generated by elements of degree $2$ as an ideal of $S_R$. 
\end{theorem}

The commutativity of the above diagram was proved  
in \cite{1905.05545} by H. Charalambous, K. Karagiannis and the first author. In order to prove theorem \ref{relative-canonical-embedding} will prove some auxiliary results first. 


\begin{lemma} \label{lemma:1-proof}
    There is a set $f_1,\ldots,f_s \in S_R$ of generators of the ideal  $I_{\mathcal{X}} \lhd S_R$ so that $f_1 \otimes 1_L,\ldots, f_s \otimes 1_L\in S_L$ generate  $I_{\mathcal{X}_\eta} \lhd S_L$. 
\end{lemma}
\begin{proof}
We will start from a basis of $I_{\mathcal{X}_\eta}$.
Since $L$ is a field it follows by Petri's Theorem, that there are elements $\tilde{f_1},\dots,\tilde{f_r}\in S_L$ of degree $2$ or $3$ such that
$I_{\mathcal{X}_{\eta}}= \langle \tilde{f_1},\dots,\tilde{f_r} \rangle$. We can find an element $c\in R$ such that $f_i\defeq c\tilde{f}_i\in S_R$ for all $i$,  $\mathrm{deg}(f_i)=\mathrm{deg}(\tilde{f_i})$ and
\[
\left< f_1\otimes 1_L, \dots, f_r\otimes 1_L\right> = \left<\tilde f_1, \dots, \tilde f_r\right>=I_{\mathcal{X}_\eta}.
\]
Let $I=\left<f_1,\dots, f_r\right> \lhd S_R$, we aim to prove that $I=I_{\mathcal{X}}$. Clearly $I_{\mathcal{X}}\otimes S_L\subseteq I_{\mathcal{X}_\eta}$ and hence 
\[
    I_{\mathcal{X}} \subseteq \left(I_{\mathcal{X}}\otimes S_L\right)^c \subseteq \left(I_{\mathcal{X}_\eta}\right)^c=I,
\]
where $\alpha^c\subseteq S_R$ is the contraction of an ideal $\alpha\subseteq S_L$, i.e. it's inverse image via the map $(-)\otimes_R L: S_R\rightarrow S_L$. For the reverse inclusion, let $a=\sum_{i=1}^{r}a_if_i$ be an arbitrary element in $I$. We will show that $a\in I_{\mathcal{X}}$. Indeed, using the commuting upper square of diagram \ref{gen-diagram} every element
$a=\sum_{\nu=1}^r a_i f_i \in I$ maps to $\sum_{\nu=1}^r a_i f_i\otimes_R1_L$ which in turn maps to $0$ by $\phi_\eta$. The same element maps to $\phi(a)$ and $\phi(a)\otimes_R 1_L$ should be zero. Since all modules $H^0(\mathcal{X},\Omega_{\mathcal{X}/R}^{\otimes n})$ are free $\phi(a)=0$ and $a\in I_{\mathcal{X}}$.
\end{proof}

\begin{lemma}
    \label{lemma:2-proof}
The quadratic generators of $I_{\mathcal{X}_0}$ can be lifted to quadratic polynomials in $S_R$ inside $I_{\mathcal{X}}$. 
\end{lemma}
\begin{proof}
Let $\bar g$ be an element of degree $2$ in $I_{\mathcal{X}_0}$, we will prove that we can select an element $g\in I_{\mathcal{X}}$ such that $g\otimes 1_k=\bar g$, so that $g$ has degree $2$. 

Let us choose a lift $\tilde{g} \in S_R$ of degree $2$ by lifting each coefficient of $\bar{g}$ from $k$ to $R$. This element is not necessarily in $I_{\mathcal{X}}$. We have $\phi(\tilde g)\otimes1_k =\phi_0(\bar{g})=0$. 
Let $\bar{e}_1,\ldots, \bar{e}_{3g-3}$ be generators of the free $R$-module $H^0(\mathcal{X},\Omega_{\mathcal{X}/R}^{\otimes 2})$ and choose 
$e_1,\ldots, e_{3g-3} \in S_R$, such that $\phi(e_i)=\bar{e}_i$ { and $\deg(e_1)=\cdots = \deg(e_{3g-3})=2$}.
Let us write $\phi(\tilde{g})=\sum_{i=1}^{3g-3} \lambda_i \bar{e}_i$, 
with $\lambda_i \in R$. 
Since $\phi_0(\bar{g})=0$ we have that all $\lambda_i \in \mathfrak{m}_R$ for all $1\leq i \leq {3g-3}$. This means that the element 
$g=\tilde{g}-\sum_{i=1}^{3g-3} \lambda_i e_i \in S_R$ reduces to $\bar{g}$ modulo $\mathfrak{m}_R$ and also 
$\phi(g)=\phi(\tilde{g})- \sum_{i=1}^{3g-3} \lambda_i \bar{e}_i=0$, so $g\in I_{\mathcal{X}}$.

Let $\bar g_1,\dots,\bar g_s\in I_{\mathcal{X}_0}$ be elements of degree $2$ such that 
\begin{equation*}
I_{\mathcal{X}_0} = \langle \bar g_1, \dots, \bar g_s\rangle.
\end{equation*}
Using the previous construction, we take the  lifts $g_1,\ldots,g_s$ in $I_\mathcal{X} \lhd S_R$, i.e. such that $g_i\otimes 1_k=\bar g _i$ with $\deg g_i=2$.  

\end{proof}

\begin{lemma} \label{lemma:independent}
Let $\bar{v}_1,\ldots,\bar{v}_n\in k^m$ be linear independent elements  and $v_1,\ldots,v_n$ be lifts in $R^m$. Then 
\[
\sum_{\nu=1}^n a_\nu v_\nu =0 \qquad a_\nu \in R,
\]
implies that $a_1=\cdots=a_n=0$. 
\end{lemma}
\begin{proof}
Since the elements $\bar{v}_1, \ldots, \bar{v}_n$ are linear independent we have 
 $n \leq m$. 
We write the elements $v_1,\ldots,v_n$ (resp. $\bar{v}_1,\ldots,\bar{v}_n$) as columns and in this way we obtain an $m\times n$ matrix $J$ (resp. $\bar{J}$). Since the elements are linear independent in $k^m$ there is an $n\times n$ minor matrix of $\bar{J}$ with an invertible determinant. Without loss of generality, we assume that there is an $n\times n$ invertible matrix $\bar Q$ with coefficients in $k$ such that 
 $\bar Q\cdot \bar{J}^t=\left(
\begin{array}{l|l}
\mathbb{I}_n  &  \bar A
\end{array}
\right)$, where $\bar{A}$ is an $(m-n) \times n$ matrix. We now get lifts $Q,J$ and $A$ of $\bar Q, \bar J$ and $\bar A$ respectively, with coefficients in R, i.e. 
\[
Q\cdot J^t\equiv 
(\begin{array}{l|l}
\mathbb{I}_n  &  A
\end{array})
\mathrm{mod}\mathfrak{m}_R.\] 
The columns $v_1,\ldots, v_n$ of J are lifts of the elements $\bar{v}_1,\ldots,\bar{v}_n$.
It follows that $Q \cdot J^t=\left(
\begin{array}{c|c}
\mathbb{I}_n  &  A
\end{array}
\right)
+ \left(
\begin{array}{c|c}
C & D
\end{array}
\right)$, where $C,D$ are matrices with entries in $\mathfrak{m}_R$. The determinant of $\mathbb{I}_n+C$ is $1+m$, for some element $m\in \mathfrak{m}_R$, and this is an invertible element in the local ring $R$. Similarly, the matrix  $Q$ is invertible, since its determinant is $\det(\bar{Q})+m'$, $m' \in \mathfrak{m}_R$.  
Therefore,  
\[
 J^t= 
\left(
\begin{array}{l|l}
Q^{-1} (\mathbb{I}_n+C)  &  Q^{-1} (A+D)
\end{array}
\right)
\]
has the first $n\times n$ block matrix invertible and the desired result follows. 

\end{proof}

\begin{remark} \label{symmetrization}
It is clear that over a ring where $2$ is invertible,  there is an 1-1 correspondence between symmetric $g\times g$ matrices and quadratic polynomials.
Indeed, a quadratic polynomial can be written as 
\[
f(w_1,\ldots,w_g) =\sum_{1\leq i,j \leq g} a_{ij} w_i w_j =w^t A w,
\] 
where $A=(a_{ij})$. 
Even if the matrix $A$ is not symmetric, the matrix $(A+A^t)/2$ is and generates the same quadratic polynomial
\[
w^t A w= w^t\left( \frac{A+A^t}{2} \right) w.
\]
Notice that the map
\[
A \mapsto \frac{A+ A^t}{2}
\]
is  onto the space of symmetric matrices and has as kernel the space of antisymmetric matrices.

\end{remark}

As a corollary of lemma \ref{lemma:independent} we obtain:
\begin{lemma}
    \label{lemma:3-proof}
    The lifts of quadratic generators are $R$-linearly independent elements in the free $R$-module of $g\times g$ symmetric matrices.
\end{lemma}

Using Lemma 5 (ii) of \cite{1905.05545} we arrive at the  following criterion
\begin{lemma} 
\label{generators}
Let $J$ be a set of polynomials in $S_R$ such that  $\langle J \rangle\otimes_R L = I_{\mathcal{X}_{\eta}}$ and $\langle J \rangle\otimes_R k = I_{\mathcal{X}_0}$. Then $I_\mathcal{X}=\langle J \rangle$.

\end{lemma}

\begin{proof}(of theorem \ref{relative-canonical-embedding})
The ideal $I_{\mathcal{X}_0}$, is known to be generated by the quadratic polynomials $\bar{g}_1,\ldots, \bar{g}_s$. 
Thus every cubic polynomial $c \in I_{\mathcal{X}_0}$ is generated by the quadratic polynomials 
$\bar{g}_1,\ldots,\bar{g}_s$, and is a linear combination of elements $w_j \bar{g}_i$, $1\leq j \leq g$, $1\leq i \leq s$. The lemma of Nakayama for local rings implies that  the $R$-module of elements in $I_\mathcal{X}$ of degree $3$ is generated by $W_j g_i$, $1\leq j \leq g$,  $1\leq i \leq s$. This means that both $I_{\mathcal{X}_\eta}$ and $I_{\mathcal{X}}$ do not contain cubic generators and are generated by quadratic polynomial as well.

By the general theory of Betti tables we know that in the cases the canonical ideal is generated by quadratic polynomials, the dimension of { the vector space spanned by the $A_i$} equals $\binom{g-2}{2}$, see \cite[prop. 9.5]{MR2103875}. { A minimal set of quadratic generators of $I$} is given by a set of polynomials $f_1,\ldots,f_r$, with $f_i=w^t A_i w$, where the symmetric polynomials are linearly independent. Consider 
\begin{itemize}
\item the $k$-vector space $(I_{\mathcal{X}_0})_2$, of degree $2$ elements of $I_{\mathcal{X}_0}$.
\item the $L$-vector space $(I_{\mathcal{X}_\eta})_2$, of degree $2$ elements of $I_{\mathcal{X}_\eta}$
\end{itemize}
We begin on the special fibre with the $s=\binom{g-2}{2}$ generators
$\bar{g}_1,\ldots,\bar{g}_s$  of $I_{\mathcal{X}_0}$ and notice that these elements form a $k$-linear base of the $s$-dimensional space $(I_{\mathcal{X}_0})_2$.

Using lemma \ref{lemma:2-proof} 
 we can lift them to  elements $J=\{g_1,\ldots,g_s\} \subset I_{\mathcal{X}}$ 
which are $R$-linear independent by lemma \ref{lemma:independent}, therefore these elements give rise to $L$-linear independent elements $g_1\otimes 1_L, \ldots, g_s \otimes 1_L$, which have the correct dimension, equal to the Betti number $\beta_{1,2}$ of the generic fibre. Recall that the $\beta_{1,2}$ is the dimension of the space of quadratic generators and is 
equal  to $\binom{g-2}{2}$, see see remark \ref{Rem5}.  This means that  $\mathrm{span}_L\{g_1\otimes 1_L,\dots, g_s\otimes 1_L\}=(I_{\mathcal{X}_\eta})_2$ and hence $I_{\mathcal{X}_\eta}=\left<g_1\otimes 1_L,\dots, g_s\otimes 1_L\right>$.

Therefore
\begin{itemize}
	\item[$(i)$] $\langle J \rangle \otimes_R L = I_{\mathcal{X}_{\eta}}$.
	\item[$(ii)$] $\langle J \rangle \otimes_R k = I_{\mathcal{X}_0}$.
\end{itemize}
and the desired result follows by lemma \ref{generators}.
\end{proof}

Essential for the proof of lemma \ref{generators} was that the ring $R$ has a generic fibre. The deformation theory is concerned with deformations over local Artin algebras which do not have generic fibres. 

{
\begin{corollary}\label{embeddedCase} Let $A$ be a local Artin algebra. By tensoring with $A$ in the middle sequence of eq. (\ref{gen-diagram}) we have the following diagram:
\[
\xymatrix{
0 \ar[r] &
	 I_{X_A}\ar@{^{(}->}[r]  
	  \ar@{->>}[d]^{\otimes_A A/\mathfrak{m}_A}
	  &
	S_A:=A[W_1,\ldots,W_g] \ar@{->>}[r]^-{\phi} 
	 \ar@{->>}[d]^{\otimes_A A/\mathfrak{m}_A}
	& 
	\displaystyle\bigoplus_{n=0}^\infty
	H^0(X_A,\Omega_{X_A/A}^{\otimes n}) \ar[r] 
	\ar@{->>}[d]^{\otimes_A A/\mathfrak{m}_A}
	& 0 
	\\
	0 \ar[r] & I_{\mathcal{X}_0}\ar@{^{(}->}[r] & S_k:=k[w_1,\ldots,w_g] \ar@{->>}[r]^-{\phi_0} &
	\displaystyle\bigoplus_{n=0}^\infty H^0(\mathcal{X}_0,\Omega_{\mathcal{X}_0/k}^{\otimes n}) \ar[r] & 0 	
}
\]
In the above, each row is exact and each square is commutative. Moreover, the ideal $I_{X_A}$ is generated by elements of degree $2$ as an ideal of $S_R$.
\end{corollary}
\begin{proof}
Since $H^0(\mathcal{X},\Omega_{\mathcal{X}/A}^{\otimes n})$ is free the left top arrow in the above diagram is injective and the images of the generators of $I_{\mathcal{X}}$ are generators of $I_{X_A}$ of degree $2$.
\end{proof}
}

{
\begin{remark}
    The above corollary provides a proof of proposition \ref{red-new-prot} in the special case of a deformation embedded in the relative projective space. In the next section we will prove that we can consider  embedded deformations without loss of generality. 
\end{remark}
}

\subsection{Embedded deformations}

Let $Z$ be a scheme over $k$ and let  $X$ be a closed subscheme of $Z$. An embedded deformation $X'\rightarrow \Spe k[\epsilon]$ of $X$ over $\Spe k[\epsilon]$ is a closed subscheme $X' \subset Z'=Z \times \Spe k[\epsilon]$  
fitting in the diagram:
\[
\xymatrix{
	& Z \ar[rr] \ar[dd] & & Z \times \Spe k[\epsilon]  \ar[dd]
	\\
X \ar[rr] \ar@{^{(}->}[ur] \ar[dr] & & X' \ar@{^{(}->}[ru] \ar[dr] &
 \\
 & \Spe k \ar[rr] & & \Spe k[\epsilon]
}
\]
Let $\mathcal{I}$ be the ideal sheaf describing $X$ as a closed subscheme of $Z$ and 
\begin{equation}
\label{sheaf}
\mathcal{N}_{X/Z}=\HomC_Z(\mathcal{I},\mathcal{O}_X)=
\HomC_X(\mathcal{I}/\mathcal{I}^2,\mathcal{O}_X), 
\end{equation}
be the normal sheaf. 
In particular, for an affine open set $U$ of $X$ we set
$B'=\mathcal{O}_{Z'}(U)= B\oplus \epsilon B $, where 
$B=\mathcal{O}_Z(U)$ and we observe that describing the sheaf of ideals $\mathcal{I}'(U) \subset \mathcal{B}'$ is equivalent to giving an element
\[
\phi_U\in 
\mathrm{Hom}_{\mathcal{O}_Z(U)}
\big(
\mathcal{I}(U),\mathcal{O}_Z(U)/\mathcal{I}(U)
\big),
\]
 see \cite[prop. 2.3]{MR2583634}. 

 In this article, we will  take $Z=\mathbb{P}^{g-1}$ and consider the canonical embedding 
$f:X \rightarrow \mathbb{P}^{g-1}$. We   will denote by $N_f$ the sheaf $\mathcal{N}_{X/\mathbb{P}^{g-1}}$.
 Let $\mathcal{I}_X$ be the sheaf of ideals of the curve $X$ seen as a subscheme of $\mathbb{P}^{g-1}$. Since the curve $X$ satisfies the conditions of Petri's theorem, it is fully described  by certain quadratic polynomials $f_1=\tilde{A}_1,\ldots,f_r=\tilde{A}_r$ which correspond to a set $g\times g$
matrices $A_1,\ldots,A_r$, see  \cite{MR4333646}.
The elements $f_1,\ldots,f_r$ generate the ideal $I_X$ corresponding to the affine cone $C(X)$ of $X$, $C(X) \subset \mathbb{A}^{g}$.
{
M. Schlessinger in \cite{MR344519} observed that   the deformations of the projective variety are related to the deformations of the affine cone.  Notice that in our case all relative projective curves are smooth and the assumptions of \cite[th. 2]{MR344519} are satisfied. We can thus replace the sheaf theoretic description of eq. (\ref{sheaf}) and work with the affine cone instead. 
}

We have
\[
H^0(X,N_f)= \mathrm{Hom}_{S}(I_X, \mathcal{O}_X),
\]
{
where $S=S_k$ { is the the symmetric algebra $k[\omega_1,\ldots,\omega_g]$ as we defined it in definition \ref{1stDef}}.
}

Assume that $X$ is deformed to a curve $X_\Gamma \rightarrow \Spe \Gamma$, where $\Gamma$ is a local Artin algebra, $X_\Gamma \subset \mathbb{P}^{g-1}_{\Gamma}=\mathbb{P}^{g-1} \times \Spe \Gamma$. 
Our initial curve $X$ is described in terms of the homogeneous canonical ideal $I_X$, generated by the elements $\{w^t A_1 w,\ldots, w^t A_r w\}$. 
For a local Artin algebra $\Gamma$ let $\mathcal{S}_g(\Gamma)$ denote the space of symmetric $g\times g$ matrices with coefficients in $\Gamma$. 
The deformations  $X_\Gamma$ are expressed in terms of the ideals $I_{X_\Gamma}$,  which by the relative Petri's theorem  are also generated by elements 
 $w^t A_1^{\Gamma} w,\ldots, w^t A_r^{\Gamma} w$,  where $A_i^{\Gamma}$ is in $\mathcal{S}_g(\Gamma)$.

\begin{remark}
A  set of quadratic generators $\{w^t A_1 w,\ldots, w^t A_r w\}$ is a minimal set of generators if and only if the elements $A_1,\ldots,A_r$ are linear independent in the free $\Gamma $-module $\mathcal{S}_g(\Gamma)$ of rank $(g+1)g/2$. 
\end{remark}

\subsubsection{Embedded deformations and small extensions}
Let 
\[
0 \rightarrow \langle E \rangle
\rightarrow
\Gamma'
\stackrel{\pi}{\longrightarrow} 
\Gamma 
\rightarrow 
0
\]
be a small extension 
and a curve $\mathbb{P}^{g-1}_{\Gamma'} \supset X_{\Gamma'}\rightarrow \Spe \Gamma'$ be a deformation of $X_\Gamma$ and $X$. 
The curve $X_{\Gamma'}$ is described in terms of quadratic polynomials $w^t A_i^{\Gamma'} w$, where $A_i^{\Gamma'} \in \mathcal{S}_g(\Gamma')$, which reduce to $A_i^{\Gamma}$ modulo $\langle E \rangle$. This means that  
\begin{equation}
\label{quad-relgen}
A_i^{\Gamma'} \equiv A_i^{\Gamma} \mod \; \mathrm{ker}(\pi) \text{ for all  } 
1 \leq i \leq r
\end{equation}
and if we select a naive lift 
 $i(A_i^{\Gamma})$ of $A_i^{\Gamma}$, then we can write 
\[
A_i^{\Gamma'}=i(A_i^{\Gamma})+E \cdot B_i, \text{ where } B_i\in \mathcal{S}_g(k). 
\]
The set of liftings of elements $A_i^{\Gamma'}$  of elements 
$A_i^{\Gamma}$, 
for $1\leq i  \leq r$
 is a principal homogeneous space, under the action of $H^0(X,N_f)
 $, since two such liftings
$\{A_i^{(1)}(\Gamma'), 1\leq i \leq r\}$, 
$\{A_i^{(2)}(\Gamma'), 1\leq i \leq r\}$ differ by a set of matrices in 
$\{B_i(\Gamma')=A_i^{(1)}(\Gamma')-A_i^{(2)}(\Gamma'), 1\leq i \leq r\}$ with entries in $\langle E \rangle \cong k$,  see also \cite[thm. 6.2]{MR2583634}.

Define a map $\phi:\langle A_1,\ldots,A_r \rangle \rightarrow \mathcal{S}_g(k)$ by  $\phi(A_i)=B_i(\Gamma')$ and we also define the corresponding map on polynomials  
$
\tilde{\phi}(\tilde{A_i}) = w^t \phi(A_i) w.
$
we obtain a map
$\tilde{\phi} \in \mathrm{Hom}_{S}(I_X, \mathcal{O}_X)=H^0(X,N_f)$, 
see also \cite[th. 6.2]{MR2583634}.
Obstructions to such liftings are known to reside in  $H^1(X,\mathcal{N}_{X/\mathbb{P}^{g-1}} \otimes_k \ker \pi)$, which we will prove it is zero, see remark \ref{some-zero-1}.

\subsubsection{Embedded deformations and tangent spaces}
Let us consider the $k[\epsilon]/k$ case. 
Since $i:X\hookrightarrow \mathbb{P}^{g-1}$ is non-singular we have the following exact sequence
\[
0 \rightarrow \mathcal{T}_X \rightarrow 
i^*\mathcal{T}_{\mathbb{P}^{g-1}}
\rightarrow
\mathcal{N}_{X/\mathbb{P}^{g-1}}
\rightarrow 0
\]
which gives rise to 
\[
\xymatrix@C=1pc{
0 \ar[r] &
 H^0(X,\mathcal{T}_X)
\ar[r] &
H^0(X,i^* \mathcal{T}_{\mathbb{P}^{g-1}})
\ar[r] &
H^0(X,\mathcal{N}_{X/\mathbb{P}^{g-1}})
\ar@{->} `r/8pt[d] `/10pt[l] `^dl[lll]^{\delta}  `^r/3pt[dlll] [dlll]
	\\
& \!\!\!\!\!\!\!\!\!\!\!\!\!\!\!\!\!\!
H^1(X,\mathcal{T}_X) \ar[r] & 
 H^1(X,i^* \mathcal{T}_{\mathbb{P}^{g-1}}) \ar[r] & 
H^1(X,\mathcal{N}_{X/\mathbb{P}^{g-1}}) \ar[r]
  & 0
 }
\]
\begin{remark}
\label{some-zero-1}
In the above diagram, the last entry in the bottom row is zero since it corresponds to a second cohomology group on a curve. 
By Riemann-Roch theorem we have that $H^0(X,\mathcal{T}_X)=0$ for $g\geq 2$. Also, the relative Petri theorem implies that the map $\delta$ is onto.
We will give an alternative proof that $\delta$ is onto by proving that $H^1(X,i^* \mathcal{T}_{\mathbb{P}^{g-1}})=0$. This proves that $H^1(X,\mathcal{N}_{X/\mathbb{P}^{g-1}})=0$ as well, so there is no obstruction in lifting the embedded deformations. 
\end{remark}
 
Each of the above spaces has a deformation theoretic interpretation, see \cite[p.96]{HarrisModuli}:
\begin{itemize}
\item The space $H^0(X,i^* \mathcal{T}_{\mathbb{P}^{g-1}})$ is the space of deformations of the map $i:X  \hookrightarrow \mathbb{P}^{g-1}$, that is both $X,\mathbb{P}^{g-1}$ are trivially deformed, see \cite[p. 158, prop. 3.4.2.(ii)]{MR2247603}
\item The space $H^0(X,\mathcal{N}_{X/\mathbb{P}^{g-1}})$ is the space of embedded deformations, where $\mathbb{P}^{g-1}$ is trivially deformed
 see \cite[p. 13, Th. 2.4)]{MR2583634}.
\item The space $H^1(X,\mathcal{T}_X)$ is the space of all deformations of $X$.
\end{itemize}
The dimension of the space $H^1(X,\mathcal{T}_X)$ can be computed using Riemann-Roch theorem on the dual space $H^0(X,\Omega_X^{\otimes 2})$ and equals $3g-3$. In next section we will give a linear algebra interpretation for the spaces $H^0(X,\mathcal{N}_{X/\mathbb{P}^{g-1}})$, 
$H^0(X,i^* \mathcal{T}_{\mathbb{P}^{g-1}})$ allowing us to compute its dimensions. 
\subsection{Some matrix computations}
We begin with the Euler exact sequence (see.  \cite[II.8.13]{Hartshorne:77},
\cite[p. 581]{VakilSea} and  \cite{MO5211}
\href{https://mathoverflow.net/questions/5211/geometric-meaning-of-the-euler-sequence-on-mathbbpn-example-8-20-1-in-ch}{MO})
\[
0 \rightarrow \mathcal{O}_{\mathbb{P}^{g-1}}\rightarrow \mathcal{O}_{\mathbb{P}^{g-1}}(1)^{\oplus g}
\rightarrow \mathcal{T}_{\mathbb{P}^{g-1}} \rightarrow 0.
\]
We restrict this sequence to the curve $X$:
\[
0 \rightarrow \mathcal{O}_X \rightarrow 
 i^* \mathcal{O}_{\mathbb{P}^{g-1}}(1)^{\oplus g}=\omega_X^{\oplus g}
\rightarrow i^* \mathcal{T}_{\mathbb{P}^{g-1} }\rightarrow 0.
\]
We now take the long exact sequence in cohomology
{\tiny
\begin{equation}
\label{long-exact-Euler}
\xymatrix{
	0 \ar[r] 
	& k=H^0(X,\mathcal{O}_X) \ar[r]|-{f_1} 
	& H^0(X, i^* \mathcal{O}_{\mathbb{P}^{g-1}}(1)^{\oplus g}) \ar[r]|-{f_2} 
	& H^0(X, i^* \mathcal{T}_{\mathbb{P}^{g-1}})
                \ar@{->} `r/8pt[d] `/10pt[l] `^dl[lll]|-{f_3}  `^r/3pt[dlll] [dlll] \\
	H^1(X,\mathcal{O}_X) \ar[r]|-{f_4}
	& H^1(X, i^* \mathcal{O}_{\mathbb{P}^{g-1}}(1)^{\oplus g}) \ar[r]|-{f_5}
	& H^1(X, i^* \mathcal{T}_{\mathbb{P}^{g-1}}) \ar[r]
	& H^2(X,\mathcal{O}_X)=0
    }
\end{equation}
}
The spaces involved above have the following dimensions:
\begin{itemize}
\item 
$i^* \mathcal{O}_{\mathbb{P}^{g-1}}(1)=\Omega_X$ (canonical bundle)
\item 
$\dim H^0(X, i^* \mathcal{O}_{\mathbb{P}^{g-1}}(1)^{\oplus g})=
g\cdot \dim H^0 (X,\Omega_X)=g^2$
\item
$\dim H^1(X,\mathcal{O}_X)=\dim H^1(X,\Omega_X)=g$
\item 
$\dim H^1(X, i^* \mathcal{O}_{\mathbb{P}^{g-1}}(1)^{\oplus g})=
g\cdot \dim H^0(X,\mathcal{O}_X)=g $
\end{itemize}
We will return to the exact sequence given in eq. (\ref{long-exact-Euler})
and the above dimension computations in the next section. 

\subsubsection{ Study of  $H^0(X,N_f)$}
By relative Petri theorem the elements $\phi(A_i)$ are quadratic polynomials not in $I_X$,  that is elements in a vector space of dimension 
$(g+1)g/2- \binom{g-2}{2}=3g-3$, where $(g+1)g/2$ is the dimension of the symmetric $g\times g$ matrices and $\binom{g-2}{2}$ is the dimension of the space generated by the generators of the canonical ideal, see \cite[prop. 9.5]{MR2103875}. 

The set of matrices $\{A_1,\ldots,A_r\}$ can be assumed to be linear independent but this does not mean that an arbitrary selection of quadratic elements $\omega^t B_i \omega\in \mathcal{O}_X$ will lead to { an element in $\mathrm{Hom}_S(I_X, \mathcal{O}_X)$}. 
Indeed, 
the linear independent elements $A_i$ might satisfy some syzygies, see the following example, where the linear independent elements $x^2, xy$ generating $I_X$
\[
x^2=
\begin{pmatrix}
x & y
\end{pmatrix}^t
\begin{pmatrix}
1 & 0 
\\
0  & 0 
\end{pmatrix}
\begin{pmatrix}
x \\y 
\end{pmatrix}
\qquad
xy=
\begin{pmatrix}
x & y
\end{pmatrix}^t
\begin{pmatrix}
0 & 1/2 
\\
1/2  & 0 
\end{pmatrix}
\begin{pmatrix}
x \\y 
\end{pmatrix}
\]
satisfy the syzygy
\[
y \cdot x^2 - x \cdot xy=0.
\]
Therefore, a map  $\phi {\in \mathrm{Hom}_S(I_X, \mathcal{O}_X)}$, should be compatible 
with the syzygy.
{ 
This means that if we set 
\[
    B_1=
    \begin{pmatrix}
    a_1 & b_1 \\
    b_1 & c_1
    \end{pmatrix}
    \text{ and }
    B_2=
    \begin{pmatrix}
    a_2 & b_2 \\
    b_2 & c_2
    \end{pmatrix}
\]
then $\phi$ is defined by 
\[
    \phi(x^2)= a_1 x^2 + b_1xy +c_1 y^2\equiv c_1 y^2 \mod I_X, \phi(xy)=a_2 x^2 + b_2xy +c_2 y^2 \equiv c_2 y^2 \mod I_X. 
\]
We should also have 
\[
    0 = \phi( y \cdot x^2 - x \cdot xy)= y \phi(x^2)- x \phi(xy)=
    c_1 y^3 - c_2 x y^2 \equiv c_1 y^3 \mod I_X.
\]
Therefore $c_1=0$ and such a morphism  $\phi$ is defined by $\phi(x^2)=0$ and $\phi(xy)=c_2 y^2$. 
} 
This phenomenon is  known as the fundamental Grothendieck flatness criterion, see \cite[1.1]{MR344519}
and also \cite[lem. 5.1, p. 28]{MR2807457}. 

%
%
%
\begin{proposition}
\label{psimapiso}
The map
\begin{align*}
\psi:M_g(k) & \longrightarrow \mathrm{Hom}_S(I_X, S/I_X) = H^0(X,\mathcal{N}_{X/\mathbb{P}^{g-1}}) \\
B & \longmapsto  \psi_{B}: \omega^t A_i \omega \mapsto \omega^t (A_i B+ B^t A_i) \omega \mod I_X
\end{align*}
identifies the vector space $M_g(k)/\langle \mathbb{I}_g\rangle$ to $H^0(X,i^*\mathcal{T}_{\mathbb{P}^{g-1}}) \subset H^0(X,\mathcal{N}_{X/\mathbb{P}^{g-1}})$.
The map $\psi$
is equivariant, where $M_g(k)$ is equipped with the adjoint action, { for $\sigma\in G$ 

\[
B\mapsto \rho(\sigma) B \rho(\sigma^{-1})=\mathrm{Ad}(\sigma)B,
\]
 that is 
\[
^\sigma\psi_B=\psi_{\mathrm{Ad}(\sigma)B}.
\]
}
\end{proposition}
\begin{proof}
Recall that the space $H^0(X,i^*\mathcal{T}_{\mathbb{P}^{g-1}})$ can be identified to the space of deformations of the map $f$, where $X$, $\mathbb{P}^{g-1}$ are both trivially deformed.

Consider a map 
\[
\Psi: w_j \mapsto w_j + \epsilon \delta_j(w),
\]
where $\delta_j(w)=\sum_{\nu=1}^g b_{j,\nu} w_\nu$. The map $\Psi$ can be defined in terms of the matrix $B=(b_{j,\nu})$
\begin{equation}
\label{eq:psi-def-def}
w \mapsto w+ \epsilon B w.
\end{equation}
By \cite{MR344519} a map  $\phi\in \mathrm{Hom}_S(I_X, S/I_X)=\mathrm{Hom}_S(I_X,\mathcal{O}_X)$ 
gives rise to a trivial deformation if and only if there is a map $\psi$ as defined in eq. (\ref{eq:psi-def-def})
so that for all $\tilde{A}_i$, $1\leq i \leq r$
\begin{equation}
\label{nab-cond}
\nabla \tilde{A}_i \cdot B w = \phi(\tilde{A}_i)=
\phi(w^t A_i w) \mod I_X.
\end{equation}  
But for $\tilde{A}_i=w^t A_i w$ we compute $\nabla \tilde{A}_i= w^t A_i$, 
therefore eq. (\ref{nab-cond}) is transformed to 
\begin{equation} \label{AiB}
w^t A_i B w= w^t B_i w \mod I_X,
\end{equation}
for a symmetric $g\times g$ matrix $B_i$ in $\mathcal{S}_g(k[\epsilon])$.
 Therefore, if $2$ is invertible according to remark \ref{symmetrization} 
we replace the matrix  $A_i B$ appearing in eq. (\ref{AiB})    by the symmetric matrix $A_i B+B^tA_i$. Since we are interested in the projective algebraic set defined by homogeneous polynomials the $1/2$ factor of remark \ref{symmetrization} can be omitted. 

For every $B\in M_g(k)$  we define 
 the map 
 $\psi_B\in \mathrm{Hom}_S(I_X, S/I_X)=\mathrm{Hom}_S(I_X,\mathcal{O}_X)$
given by
\[
\tilde{A}_i=\omega^t A_i \omega \mapsto \omega^t (A_i B+ B^t A_i) \omega \mod I_X, 
\]
and we have just proved that the functions $\psi_B$ are all elements in 
$H^0(X,i^* \mathcal{T}_{\mathbb{P}^{g-1}})$. 
The kernel of the map $\psi: B \mapsto \psi_B$ consists of all matrices $B$ satisfying:
\begin{equation} \label{psi-map-def}
A_i B= -B^t A_i \mod I_X \text{ for all } 1\leq i \leq \binom{g-2}{2}.  
\end{equation}
This kernel seems to depend on the selection of the elements $A_i$,  but this is not the case.  We will prove that the kernel consists of all multiples of the identity matrix. 
Indeed, 
\[
{ \dim H^0(X, i^* \mathcal{T}_{\mathbb{P}^{g-1}}) }= g^2- \ker \psi. 
\]
We now rewrite the spaces in eq. (\ref{long-exact-Euler}) by their dimensions we get
\begin{equation*}
\xymatrix{
	(0) \ar[r] 
	& (1) \ar[r]^{f_1}
	& (g^2) \ar[r]^-{f_2}
	& (g^2-\ker \psi)
                \ar@{->} `r/8pt[d] `/10pt[l] `^dl[lll]_{f_3}  `^r/3pt[dlll] [dlll] \\
	(g) \ar[r]
	& (g) \ar[r]
	& (?) \ar[r]
	& (0)
    }
\end{equation*}
So 
\begin{itemize}
	\item $\dim \ker f_2 = \dim \Ima f_1=1$
	\item $\dim \ker f_3 = \dim \Ima f_2 = g^2-1$
	\item $\dim \Ima f_3 = (g^2-\dim \ker \psi) - (g^2-1) = 1-\dim \ker \psi$
\end{itemize}
It is immediate that $\dim \ker \psi =0 \text{ or } 1$. But obviously $\mathbb{I}_g \in \ker \psi$, and hence 
\begin{equation*}
	\dim \ker \psi=1.
\end{equation*}
Finally $\dim \Ima f_3 =0$, i.e. $f_3$ is the zero map and we get the small exact sequence,

\begin{equation*}
\xymatrix{
	0 \ar[r] 
	& k=H^0(X,\mathcal{O}_X) \ar[r]
	& H^0(X, i^* \mathcal{O}_{\mathbb{P}^{g-1}}(1)^{\oplus g}) \ar[r]
	& H^0(X, i^* \mathcal{T}_{\mathbb{P}^{g-1}}) \ar[r]
	& 0
}
\end{equation*}
It follows that
\begin{equation*}
\dim H^0(X, i^* \mathcal{T}_{\mathbb{P}^{g-1}}) = g^2-1.
\end{equation*}
We have proved that $\psi:M_g(k)/ \langle \mathbb{I}_g \rangle \rightarrow H^0(X,i^*\mathcal{T}_{\mathbb{P}^{g-1}})$ is an isomorphism of vector spaces. We will now prove it is equivariant.

Using remark \ref{action-operators} we have that the action of the group $G$ on the function  
\[
\psi_B: A_i \mapsto A_i B+B^t A_i,
\] seen as an element in $H^0(X,i^*\mathcal{T}_{\mathbb{P}^{g-1}})$ is given:
\begin{align*}
A_i &\mapsto T(\sigma^{-1}) A_i  \stackrel{\psi_B}{\longmapsto} 
T(\sigma) 
\left(
\rho(\sigma)^t A_i \rho(\sigma) B + B^t \rho(\sigma)^t A_i \rho(\sigma)
\right)
\\
&= 
\left(
A_i \rho(\sigma)  B \rho(\sigma^{-1}) + (\rho(\sigma) B \rho(\sigma^{-1}))^t A_i
\right)
\end{align*}
\end{proof}

{
From now on, we will denote by $\psi$ the isomorphism $M_g(k)/\left<\mathbb{I}_g\right>\rightarrow H^0(X,\mathcal{N}_{X/\mathbb{P}^{g-1}})$ induced by $\psi$.
}

\begin{corollary}
The space $H^0(X,i^* \mathcal{T}_{\mathbb{P}^{g-1}})^G$ is generated by the elements $B\neq \{ \lambda\mathbb{I}_g: \lambda \in k\}$ such that
\[
\rho(\sigma) B \rho(\sigma^{-1}) B^{-1}=[\rho(\sigma),B]
 \in \langle A_1,\ldots,A_r \rangle
 \text{ for all } \sigma\in (X). 
\]
\end{corollary}
\begin{remark}
This construction
allows us to compute the space $ H^1(X, i^* \mathcal{T}_{\mathbb{P}^{g-1}})$. Indeed, we know that $f_4$ is isomorphism and hence $f_5$ is the zero map, on the other hand $f_5$ is surjective, it follows that $H^1(X, i^* \mathcal{T}_{\mathbb{P}^{g-1}})=0$. This provides us with another proof of the exactness of the sequence  

\begin{equation}
\label{co-ses}
\xymatrix{
0 \ar[r] &
H^0(X,i^* \mathcal{T}_{\mathbb{P}^{g-1}})
\ar[r] &
H^0(X,\mathcal{N}_{X/\mathbb{P}^{g-1}})
 \ar[r]^-{\delta}  & 
H^1(X,\mathcal{T}_X) \ar[r] & 0 
 }
\end{equation}
\end{remark}

\begin{proof}(of proposition \ref{red-new-prot}) Consider a ring that satisfies the conditions of Proposition 2 and a deformation over this ring. We have just proven that every deformation can be treated as an embedded deformation. Corollary \ref{embeddedCase} now completes the proof.
\end{proof}

\subsection{Invariant spaces}
Let 
\[
0 \rightarrow A \rightarrow B \rightarrow C \rightarrow 0
\]
be a short exact sequence of $G$-modules. We have the following sequence of $G$-invariant spaces
\[
0 \rightarrow A^G \rightarrow B^G \rightarrow C^G 
\stackrel{\delta_G}{\longrightarrow} H^1(G,A)
\rightarrow \cdots
\] 
where the map $\delta_G$ is computed as follows: an  element $c$ is given as a class $b \mod A$ and it is invariant if and only if 
$g b-b=a_g \in A$. The map $G \ni g \mapsto a_g$ is the cocycle defining 
$\delta_G(c) \in H^1(G,A)$. 

Using this construction on the short exact sequence of eq. (\ref{co-ses}) we arrive at  
\[
\xymatrix@C=0.9pc{
0 \ar[r] &
H^0(X,i^* \mathcal{T}_{\mathbb{P}^{g-1}})^G
\ar[r] &
H^0(X,\mathcal{N}_{X/\mathbb{P}^{g-1}})^G
 \ar[r]^-{\delta}  & 
H^1(X,\mathcal{T}_X)^G 
\ar@{->} `r/8pt[d] `/10pt[l] `^dl[lll]^{\delta_G}  `^r/3pt[dlll] [dlll] \\
& 
\!\!\!\!\!\!\!\!\!\!
H^1
\big(
G,H^0(X,i^* \mathcal{T}_{\mathbb{P}^{g-1}})
\big) \ar[r] & \cdots & & &
 }
\]
We will use eq. (\ref{co-ses}) in order to represent elements in $H^1(X,\mathcal{T}_X)$ as elements $[f] \in H^0(X,\mathcal{N}_{X/\mathbb{P}^{g-1}}) / H^0(X,i^* \mathcal{T}_{\mathbb{P}^{g-1}})=H^0(X,\mathcal{N}_{X/\mathbb{P}^{g-1}}) / \mathrm{Im}\psi$. 

\begin{proposition}
\label{prop12}
Let $[f]\in H^1(X,\mathcal{T}_X)^G$ be a class of a map $f:I_X\rightarrow S/I_X$ modulo $\mathrm{Im}\psi$. 
For each element $\sigma\in G$ there is a matrix
 $B_\sigma[f]$, depending on $f$, which defines  a class in $ M_g(k)/\langle\mathbb{I}_g\rangle$  satisfying the cocycle condition in eq. (\ref{coc-cond}), such that 
\[
\delta_G(f)(\sigma):A_i\mapsto A_i 
\left( B_\sigma[f] \right) + \left(B_\sigma^t[f]  \right)A_i \mod \langle A_1,\ldots,A_r \rangle.
\]
\end{proposition}
\begin{proof}
Let $[f] \in H^1(X,\mathcal{T}_X)^G$, where 
 $f:I_X\rightarrow S/I_X$ that is  $f\in H^0(X,\mathcal{N}_{X/\mathbb{P}^{g-1}})$. The $\delta_G(f)$ is represented by an $1$-cocycle given by 
 $\delta_G(f)(\sigma)=^\sigma\!\!f-f$.
Using the equivariant isomorphism of $\psi:M_g(k)/\langle \mathbb{I}_g\rangle \rightarrow H^0(X,i^*\mathcal{T}_{\mathbb{P}^{g-1}})$ {of proposition \ref{psimapiso}} we arrive at the diagram:
\[
\xymatrix@R-15pt{
G \ar[r] 
&  H^0(X,i^* \mathcal{T}_{\mathbb{P}^{g-1}}) \ar[r]^{\psi^{-1} } &
 M_g(k)/\langle \mathbb{I}_g \rangle 
\\
\sigma \ar@{|->}[r]
&  \delta_G(f)(\sigma)
\ar[r]
&
B_\sigma[f] := \psi^{-1}(\delta_G(f)(\sigma))
}
\]
By definition $B_{\sigma}[f]$ satisfies,
\begin{align}
\label{eq17}
\delta_G(f)(\sigma)( A_i) &=
 A_i B_\sigma[f] + B_\sigma[f]^t A_i \mod I_X.
\end{align}
If we denote by $\overline B_\sigma[f] \in M_g(k)$ a matrix that represent the class of $B_\sigma[f] \in M_g(k)/\mathbb{I}_g$ we have
\begin{align}
\label{coc-cond}
\overline B_{\sigma \tau}[f] & =\overline B_\sigma[f]+ \sigma \overline B_\tau[f] \sigma^{-1}
+\lambda(\sigma,\tau)\mathbb{I}_g \\
 &= \overline B_\sigma[f] + \mathrm{Ad}(\sigma) \overline B_\tau[f]+\lambda(\sigma,\tau)\mathbb{I}_g,
\nonumber
\end{align}
for all $\sigma,\tau\in G$.
In the above equation we have used the fact that $\sigma\mapsto B_\sigma[f]$ is a $1$-cocycle in the quotient space $M_g(k)/\mathbb{I}_g$, therefore the cocycle condition holds up to an element of the form $\lambda(\sigma,\tau) \mathbb{I}_g$.

\end{proof}

\begin{remark}
{
We can easily obtain that
\[
    B_1[f]=\psi^{-1}\left(\delta(f)(1)\right)=\psi^{-1}(^1f-f)=\psi^{-1}(0)=0\in M_g(k)/\left<I_X\right>.
\]
}
\end{remark}

\begin{lemma}
Let
\[
\lambda(\sigma,\tau)
\mathbb{I}_g=B_{\sigma \tau}[f]-B_{\sigma}[f]-\rm{Ad}(\sigma)B_{\tau}[f].
\]
The map $G\times G \rightarrow k$, $(\sigma,\tau)\mapsto \lambda(\sigma,\tau)$ is a normalized 2-cocycle (see \cite[p. 184]{Weibel}), that is 
\begin{align*}
0 & =\lambda(\sigma,1) = \lambda(1,\sigma)  & \text{ for all } \sigma \in G\\
0 &= {\rm{Ad}(\sigma_1)}\lambda(\sigma_2,\sigma_3) -\lambda(\sigma_1 \sigma_2,\sigma_3)+
\lambda(\sigma_1,\sigma_2 \sigma_3) - \lambda(\sigma_1,\sigma_2) 
& \text{ for all } \s_1,\sigma_2,\sigma_3 \in G
\\
 &= \lambda(\sigma_2,\sigma_3) -\lambda(\sigma_1 \sigma_2,\sigma_3)+
\lambda(\sigma_1,\sigma_2 \sigma_3) - \lambda(\sigma_1,\sigma_2) 
& \text{ for all } \s_1,\sigma_2,\sigma_3 \in G
\end{align*}
For the last equality notice that the $\mathrm{Ad}$-action is trivial on scalar multiples of the identity. 
\end{lemma}

\begin{proof}
The first equation is clear. For the second one, 
\[
\lambda(\s_1\s_2,\s_3)\mathbb{I}_g=B_{\s_1\s_2\s_3}[f]-B_{\s_1\s_2}[f]-\rm{Ad}(\s_1\s_2)B_{\s_3}[f]
\]
and
\[
\lambda(\s_1,\s_2)\mathbb{I}_g=B_{\s_1\s_2}[f]-B_{\s_1}[f]-\rm{Ad}(\s_1)B_{\s_2}[f].
\]
Hence
\begin{align*}
\lambda(\s_1\s_2,\s_3)\mathbb{I}_g + \lambda(\s_1,\s_2)\mathbb{I}_g 
= &B_{\s_1\s_2\s_3}[f]- \mathrm{Ad}(\s_1\s_2)B_{\s_3}[f] -B_{\s_1}[f]-\mathrm{Ad}(\s_1)B_{\s_2}[f]\\
= &B_{\s_1\s_2\s_3}[f]- B_{\s_1}[f]- \mathrm{Ad}(\s_1)B_{\s_2\s_3}[f] +\\
  &+ \mathrm{Ad}(\s_1)B_{\s_2\s_3}[f]-\mathrm{Ad}(\s_1)B_{\s_2}[f]- \mathrm{Ad}(\s_1\s_2)B_{\s_3}[f]\\
= &\lambda(\s_1,\s_2\s_3)
\mathbb{I}_g
+\rm{Ad}
(\s_1)\big(B_{\s_2\s_3}[f]-B_{\s_2}[f]-
{
\mathrm{Ad}(\s_2)B_{\s_3}[f]
} \big)\\
= 
&  \mathrm{Ad}(\s_1)\lambda(\s_2,\s_3)
\mathbb{I}_g +\lambda(\s_1,\s_2\s_3)
\mathbb{I}_g.
\end{align*}
\end{proof}
{
Let us fix the following notation.
Since $T$-action respect the canonical ideal, or equivalent the vector space spanned by $A_i$, there are $\lambda_{i,\nu}(\sigma)\in k$ such that
\[
T(\sigma^{-1})(A_i)= \rho(\sigma)^t A_i \rho(\sigma) =\sum_{\nu=1}^r  \displaystyle \lambda_{i,\nu}(\sigma) A_\nu.
\]
We thus obtain:
\[
\xymatrix{
^\sigma\! f: A_i \ar[r]^-{T(\sigma^{-1})} &
  T(\sigma^{-1}) A_i 
\ar[r]^-{f}&
f (
T(\sigma^{-1}) A_i 
)
\ar[r]^-{T(\sigma)} &
T(\sigma) f (
T(\sigma^{-1}) A_i 
).
}
\]
}

\begin{corollary}
If $f(\omega^t A_i \omega)=\omega^t B_i \omega$, where $B_i\in M_g(k)$ are the images of the elements defining the canonical ideal in the small extension $\Gamma' \rightarrow \Gamma$, then the symmetric matrices defining the canonical ideal $I_X(\Gamma')$ are given by $A_i+E\cdot  B_i$. Using proposition \ref{prop12} we have
\begin{align}
\label{action-onB}
(^\sigma f -f)(A_i) & = \sum_{\nu=1}^r \lambda_{i,\nu}(\sigma ) T(\sigma) (B_\nu)-B_i
\\
&=   \left(A_i B_\sigma [f]+ B_\sigma ^t[f] A_i\right) \mod \langle A_1,\ldots,A_r \rangle \nonumber
\\
&= \psi_{ B_\sigma [f]} A_i. \nonumber 
\end{align}
Furthermore, we have 
\begin{equation}
\label{e19}
\sum_{\nu=1}^r \lambda_{i,\nu}(\sigma )  (B_\nu) -T(\sigma ^{-1}) B_i =
T(\sigma ^{-1}) \psi_{B_\sigma [f]}(A_i).
\end{equation}

\end{corollary}



%
%
\section{On the deformation theory of curves with automorphisms}
Aim of this section is to prove theorem \ref{th:main-lift}.
Let $X \rightarrow \Spe k$ be a curve satisfying the assumptions of Petri's theorem and {whose canonical ideal $I_X$} is generated by quadratic polynomials $f_1,\ldots,f_r \in S_k$.
Let $X_A \rightarrow \Spe A$ be a deformation of $X$, where $A$ is a local ring with 
$A/\mathfrak{m}_A=k$,  with canonical ideal $I_{X_A}$. 
{Using proposition \ref{red-new-prot}  we get that there are quadratic polynomials $\tilde{f}_1, \ldots, \tilde{f}_r \in S_A$ such that 
\[
I_X=\left<f_1,\dots,f_r\right> \text{ and } I_{X_A}=\left<\tilde f_1,\dots,\tilde f_r\right>.
\] 
In remark \ref{symmetrization} we saw that each polynomial corresponds to a symmetric matrix. Let $A_1, \dots, A_r\in\mathrm{M}_g(k)$ be the corresponding matrices of $f_1,\dots, f_r$ and $\tilde A_1,\dots \tilde A_r\in\mathrm{M}_g(A)$ the corresponding matrices of $\tilde f_1,\dots \tilde f_r$. 

Assume now, following the assumptions of theorem \ref{th:main-lift}, that there is a lift $\rho_A$ of the action $\rho$, i.e.

\begin{equation}
\xymatrix{
&  \GL_g(A) \ar[d]^{\mod \mathfrak{m}_A}
\\
 G \ar[r]_-{\rho} 
   \ar[ru]^{\rho_{A}}
 & \GL_g(k) 
}
\end{equation}
Following \cite[lemma 4]{MR4333646}, we have it's relative analog.
\begin{lemma}
    An element in $D\in \mathrm{GL}_g(A)$ corresponds to an automorphism $\A(X_{A})$ if and only if $D^tA_iD\in \mathrm{span}_A\{A_1\dots,A_r\}$ for all $i$.
\end{lemma}
\begin{proof}
    An element of $\mathrm{GL}_g(A)$ give rise to an element in $\mathrm{PGL}_{g}(A)$ i.e. an automorphism of $\mathbb{P}^{g-1}_A$. Since the deformation $X_A$ is embedded in $\mathbb{P}^{g-1}_A$, a matrix $D\in \mathrm{GL}_g(A)$ corresponds to an automorphism of $X_A$ if and only if it's induced element in $\mathrm{PGL}_{g}(A)$ respect the canonical ideal $I_{X_A}$. If $D\cdot \tilde A_i\in \mathrm{span}_A\{\tilde A_1\dots,\tilde A_r\}$, we have
    \[
        D\cdot A_i=\sum \lambda_j \tilde A_j, \text{ for some } \lambda_j\in A.
    \]
    Hence $D\cdot \tilde f_j\in I_{X_A}$. 
    On the other hand, if $D\cdot \tilde f_i\in I_{X_A}$, there are $\lambda_j(w)\in S_A$ such that $D\cdot \tilde f_i=\sum \lambda_j(w)\tilde f_j$. Recall now that $\mathrm{deg}\left(\tilde f_i\right)=2$, for all $i$. The action of $D$ must respect the degrees and hence $\mathrm{deg}\left(D\cdot \tilde f_i\right)=2$. Finally each $\tilde f_i=f_i \mathrm{mod} \mathfrak{m}_A$ and hence the leading coefficient of $\tilde f_i$ must be a unit in $A$. That gives $\mathrm{deg}\left(\lambda_j(w)\right)=0$ for all $j$, which complete the proof.
\end{proof}
We thus have,
\begin{corollary}
    An element $\sigma\in \A(X)$ can be lifted to an element in $\A(X_A)$ if 
    \[\rho_A(\sigma)\cdot I_{X_A}= I_{X_A}.\]
where the above action is the $T(\sigma^{-1})$-action on the matrices $\tilde A_1,\dots \tilde A_r$ (see definition \ref{action-operators})
\end{corollary}
This proves theorem \ref{th:main-lift}.
}

\label{sec:grpsDefs}
Let $1 \rightarrow \langle E \rangle \rightarrow \Gamma'\rightarrow \Gamma \rightarrow 0$ be a small extension of Artin local algebras and consider the diagram
\[
\xymatrix{
X_{\Gamma}\ar[d] \ar[r] &  X_{\Gamma'}\ar[r] \ar[d] &	\mathcal{X} \ar[d] 
	\\
\mathrm{Spec}(\Gamma) \ar[r]&	\mathrm{Spec}(\Gamma') \ar[r] & \mathrm{Spec} (R)
}
\]
{ where $R$ is the versal deformation  ring.}

Suppose that $G$ acts on $X_\Gamma$, that is every automorphism $\sigma \in G$ satisfies $\sigma(I_{X_\Gamma})=I_{X_\Gamma}$. 
If the action of the group $G$ is lifted to $X_{\Gamma'}$ then we should have a lift of the representations $\rho, \rho^{(1)}$ defined in eq. (\ref{rhodef}), (\ref{rho1def1}) to $\Gamma'$ as well. The set of all such liftings is a principal homogeneous space parametrized by the spaces
$H^1(G,M_g(k)), H^1(G,M_r(k))$, provided that the corresponding lifting obstructions in 
$H^2(G,M_g(k)), H^2(G,M_r(k))$ both vanish.

Assume that there is a  lifting of the representation
\begin{equation}
\label{eq:rho1l}
\xymatrix{
&  \GL_g(\Gamma') \ar[d]^{\mod \langle E \rangle}
\\
 G \ar[r]_-{\rho_\Gamma} 
   \ar[ru]^{\rho_{\Gamma'}}
 & \GL_g(\Gamma) 
}
\end{equation}
This lift  gives rise to a lifting of the corresponding automorphism group to  the curve $X_{\Gamma'}$ if 
\[
\rho_{\Gamma'}( \sigma) I_{X_{\Gamma'}} =I_{X_{\Gamma'}} 
\quad \text{ for all } \sigma \in G,
\]
that is if the relative canonical ideal is invariant under the action of the lifted representation $\rho_{\Gamma'}$. 
In this case the free $\Gamma'$-modules $V_{\Gamma'}$, defined in remark \ref{Rem5}, are $G$-invariant and the $T$-action, as defined in definition \ref{def:T-action}.1 restricts to  a lift of the representation 
\begin{equation}
\label{eq:rho1}
\xymatrix{
&  \GL_r(\Gamma') \ar[d]^{\mod \langle E \rangle}
\\
 G \ar[r]_-{\rho^{(1)}_\Gamma} 
   \ar[ru]^{\rho^{(1)}_{\Gamma'}}
 & \GL_r(\Gamma) 
}
\end{equation}

\begin{remark}
\label{rem:Trho1actions}
    The $T$-action on the space of symmetric $g\times g$ matrices, when restricted to the quadratic generators of the defining ideal $I_X$, is just the $\rho^{(1)}$-action defined in eq. (\ref{rho1def1}). Liftings of the representation $\rho$ induce liftings of $T$ and of $\rho^{(1)}$. 
\end{remark}

In \cite[sec. 2.2]{MR4333646} we gave an efficient way to check this compatibility in terms of linear algebra: 

Consider an ordered  basis $\Sigma$ of the free $\Gamma$-module $\mathcal{S}_g(\Gamma)$ generated by  the matrices $\Sigma(ij)=(\sigma(ij))_{\nu,\mu}$, $1\leq i \leq j \leq g$ ordered lexicographically, with elements 
\[
\sigma(ij)_{\nu,\mu}
=
\begin{cases}
 \delta_{i,\nu} \delta_{ j,\mu}+\delta_{i,\mu}\delta_{j,\nu}, 
& \text{ if } i\neq j \\
 \delta_{i,\nu} \delta_{i,\mu} & \text{ if } i=j.
\end{cases}
\]
For example, for $g=2$ we have the elements 
\[\sigma(11)=
\begin{pmatrix}
1 & 0 \\
0 & 0
\end{pmatrix} \quad
\sigma(12)=
\begin{pmatrix}
0 & 1 \\
1 & 0
\end{pmatrix} \quad
\sigma(22)=
\begin{pmatrix}
0 & 0 \\
0 & 1 
\end{pmatrix}.
\]
For every symmetric matrix $A$, let $F(A)$ be the column vector 
consisted of the coordinates of $A$ in the basis $\Sigma$. 
Consider the symmetric matrices $A_1^{\Gamma'},\ldots,A_r^{\Gamma'}$,  which exist since at the level of curves there is no obstruction of the embedded deformation. 
For each $\sigma \in G$ { let $F_{\Gamma'}(\sigma)$ be the $(g+1)g/2 \times 2r$ matrix,
\begin{equation}
\label{lift-cond}
F_{\Gamma'}(\sigma)=
\left[
  F \left(A_1^{\Gamma'}\right),\ldots,F\left(A_r^{\Gamma'}\right), 
F\left(\rho_{\Gamma'}(\sigma)^t A_1^{\Gamma'} \rho_{\Gamma'}(\sigma)\right),
\ldots, 
F\left(\rho_{\Gamma'}(\sigma)^t A_r^{\Gamma'}\rho_{\Gamma'}(\sigma)
\right)\right]
\end{equation}
}
The automorphism  $\sigma$ acting on the relative curve $X_\Gamma$ is lifted to an automorphism $\sigma$ of $X_{\Gamma'}$ if and only if the matrix given in eq. (\ref{lift-cond}) has rank $r$.  

{
\begin{proposition}
\label{prop:lift-obsturctions}
The action of $G$ on $X_\Gamma$ lifts to $X_\Gamma'$ if and only if the following two conditions are satisfied
\begin{enumerate}
\item 
The cohomology class 
\[
A(\sigma,\tau) = \rho_{\Gamma'}(\sigma) \rho_{\Gamma'}(\tau) \rho_{\Gamma'}(\sigma \tau)^{-1}
\]
in $H^2(G,M_g(k))$ vanishes
\item 
the rank of the matrix $F_{\Gamma'}(\sigma)$ equals $r$ for all elements $\sigma \in G$. 
\end{enumerate}
\end{proposition}
\begin{remark}
In proposition \ref{prop:lift-obsturctions} the first condition 
is necessary for the lifted automorphisms to form a group and the second is necessary for the automorphisms to lift at all.
\end{remark}
}
\subsection{A tangent space condition}
\label{sec:AtangenSpaceCond}
All lifts of $X_\Gamma$ to $X_{\Gamma'}$ form a principal homogeneous space under the action of 
 $H^0(X, \mathcal{N}_{X/\mathbb{P}^{g-1}})$.
This paragraph aims to provide the compatibility relation
given in eq. (\ref{eq:compat-cond}) by selecting the deformations of the curve and the representations.


Let $\{A_1^{\Gamma},\dots, A_r^{\Gamma}\}$ be a basis of the canonical Ideal $I_{X_\Gamma}$, where $X_\Gamma$ is a canonical curve. { Recall that these means that $A_i^{\Gamma}$ are lifts of the matrices $A_i$.}
Assume also that the special fibre is acted on by the group $G$, and we assume that the action of the group $G$ is lifted to the relative curve $X_{\Gamma}$.
Since $X_\Gamma$ is assumed to be acted on by $G$, we have the action 
\begin{equation}\label{lambda}
T(\sigma^{-1})(A_i^{\Gamma})=
\rho_\Gamma(\sigma)^t A_i^{\Gamma} \rho_\Gamma(\sigma)=
\sum_j \lambda_{i,j}^{\Gamma}(\sigma) { A_j^\Gamma} \text{ for each } i=1,\dots,r,
\end{equation}
where $\rho_\Gamma$ is a lift of the representation $\rho$ induced by the action of $G$ on $H^0(X_\Gamma,\Omega_{X/\Gamma})$ and $\lambda_{i,j}^{\Gamma}(\sigma)$ are the 
entries of the matrix of the lifted representation $\rho^{(1)}_\Gamma$ induced by the action of $G$ on $A_1^{\Gamma},\ldots, A_r^{\Gamma}$. Notice that the  matrix $\rho_\Gamma(\sigma) \in \GL_g(\Gamma)$. 
We will denote by $A_{1}^{\Gamma'},\ldots,A_{r}^{\Gamma'} \in \mathcal{S}_g(\Gamma')$   a set of  liftings of the matrices $A_{1}^{\Gamma},\ldots,A_{r}^{\Gamma}$.
Since the couple $(X_\Gamma,G)$ is lifted to $(X_{\Gamma'},G)$, there is an action 
\[
T(\sigma^{-1})(A_i^{\Gamma'})=
\rho_{\Gamma'}(\sigma)^t A_i^{\Gamma'} \rho_{\Gamma'}(\sigma)=
\sum_j \lambda_{i,j}^{\Gamma'}(\sigma) A_j^{\Gamma'} \text{ for each } i=1,\dots,r,
\]
where $\lambda_{ij}^{\Gamma'}(\sigma) \in \Gamma'$.
 All other liftings extending $X_\Gamma$ form a principal homogeneous space under the action of
$H^0(X, \mathcal{N}_{X/\mathbb{P}^{g-1}})$, that is we can find matrices 
$B_1,\ldots, B_r \in \mathcal{S}_g(k)$, such that the set
\[
\{A_1^{\Gamma'}+E\cdot B_1,\dots, A_r^{\Gamma'}+E\cdot B_r\}
\]
forms a basis for another lift $I_{X^{1}_{\Gamma'}}$ of the canonical ideal of $I_{X_{\Gamma}}$. This means that all lifts of the canonical curve $I_{X_\Gamma}$ differ by an element 
 $f \in \mathrm{Hom}_S(I_X, S/I_X)=H^0(X,\mathcal{N}_{X/\mathbb{P}^{g-1}})$  so that $f(A_i)=B_i$.


In the same manner, if $\rho_{\Gamma'}$ is a lift of the representation $\rho_{\Gamma}$ every other lift is given by 
\[
\rho_{\Gamma'}(\sigma) + E\cdot \tau(\sigma),
\]
where $\tau(\sigma)\in M_g(k)$. 

{
\begin{remark}\label{coc-def1-rewrite} We will rewrite lemma \ref{coc-def1} with this notation. Let $\rho_{\Gamma'}=\rho^2_{\Gamma'}$ and $\rho_{\Gamma'}+E\tau(\sigma) =\rho^1_{\Gamma'}$ for some $\tau \in \mathrm{M}_g(k)$,

\begin{align*}
d(\sigma) &=
\frac{
 \rho^1_{\Gamma'}(\sigma) 
 \rho^2_{\Gamma'} (\sigma)^{-1}
 - \mathbb{I}_{\Gamma'}
 }{E} = 
\frac{
   \big(  \rho_{\Gamma'}(\sigma) + E \tau(\sigma) \big)\rho_{\Gamma'}(\sigma)^{-1}  - \mathbb{I}_{\Gamma'} 
  }{E} 
 = \tau (\sigma)\rho_{\Gamma'}(\sigma)^{-1}.
\end{align*} 

\end{remark}

We have to find out when 
 $\rho_{\Gamma'}(\sigma)+ E\cdot \tau(\sigma)$ is  an automorphism of the relative curve $X_{\Gamma'}$, i.e. when 

\begin{equation}\label{spanR}
T(\rho_{\Gamma'}(\sigma^{-1})+E \cdot \tau(\sigma^{-1}))(A_i^{\Gamma'}+E \cdot B_i) \in \mathrm{span}_{\Gamma'}\{A_1^{\Gamma'}+ E \cdot B_1,\dots, A_r^{\Gamma'}+ E \cdot B_r\},
\end{equation}
that is 
\begin{align}
\label{spanR2}
(\rho_{\Gamma'}(\sigma) & +E \cdot  \tau(\sigma))^t
\left(
A_i^{\Gamma'} +E \cdot  B_i
\right)
(\rho_{\Gamma'}(\sigma)+E \cdot \tau(\sigma)) = 
\sum_{j=1}^r 
\tilde{\lambda}^{\Gamma'}_{ij}(\sigma)
\left( 
A_j^{\Gamma'} +E \cdot  B_j
\right),
\end{align} 
for some $\tilde{\lambda}^{\Gamma'}_{ij}(\sigma) \in \Gamma'$. 
Since
\[
T_{\Gamma'}(\sigma^{-1})A_i^{\Gamma'}
=\rho_{\Gamma}(\sigma)^t A_i^{\Gamma} \rho_{\Gamma}(\sigma) \mod \langle E \rangle
 \]
we have that $\tilde{\lambda}^{\Gamma'}_{ij}(\sigma)=\lambda^{\Gamma}_{i,j}(\sigma) \mod E$, therefore we can write 
\begin{equation}
\label{rho1def}
\tilde{\lambda}^{\Gamma'}_{ij}(\sigma)=\lambda_{ij}^{\Gamma'}(\sigma)+ E\cdot \mu_{ij}(\sigma),
\end{equation}
for some $\mu_{ij}(\sigma) \in k$.
We expand first the right-hand side of eq. (\ref{spanR2}) using eq. (\ref{rho1def}). We have
\begin{align}
\label{expandRho1}
 \sum_{j=1}^r \tilde{\lambda}^{\Gamma'}_{ij}(\sigma)
\left( 
A_j^{\Gamma'} +E \cdot  B_j
\right)
&=
\sum_{j=1}^r 
\left(
\lambda_{ij}^{\Gamma'}(\sigma)+ E\cdot \mu_{ij}(\sigma)
\right)
\left( 
A_j^{\Gamma'} +E \cdot  B_j
\right)\\
&=
\sum_{j=1}^r
\lambda_{ij}^{\Gamma'}(\sigma) A_j^{\Gamma'}
+E 
\big(
\mu_{ij}(\sigma) A_j + \lambda_{ij}(\sigma) B_j
\big).
\end{align} 
Here we have used the fact that $E \mathfrak{m}_\Gamma= E \mathfrak{m}_{\Gamma'}$ so  $E \cdot x= E \cdot (x \mod \mathfrak{m}_{\Gamma'})$ for every $x\in \Gamma'$.

We now expand  the left-hand side of eq. (\ref{spanR2}). 
\begin{align*}
(\rho_{\Gamma'}(\sigma) & +E \cdot  \tau(\sigma))^t
\left(
A_i^{\Gamma'} +E \cdot  B_i
\right)
(\rho_{\Gamma'}(\sigma)+E \cdot \tau(\sigma)) = 
\rho_{\Gamma'}(\sigma)^t A_i^{\Gamma'} \rho_{\Gamma'}(\sigma)
\\
 &+ 
E \cdot 
\left( 
\rho(\sigma)^t B_i \rho(\sigma) + {\tau(\sigma)^t} A_i \rho(\sigma) + \rho(\sigma)^tA_i \tau(\sigma) 
\right).
\end{align*}

Set $D_{\sigma}=  \tau(\sigma)\rho(\sigma)^{-1}=d(\sigma)$ according to the notation of remark \ref{coc-def1-rewrite}, we can write
\begin{equation} \label{B}
\begin{split}
\tau(\sigma)^t A_i \rho(\sigma)
 &+ \rho(\sigma)^t A_i \tau(\sigma)
 \\
 &=
 \rho(\sigma)^t \rho(\sigma^{-1})^t\tau(\sigma)^t   
 A_i \rho(\sigma) + \rho(\sigma)^t A_i \tau(\sigma)    \rho(\sigma)^{-1} \rho(\sigma)
  \\
&= 
\rho(\sigma)^t 
(D_{\sigma}^t  A_i)  \rho(\sigma)
 + \rho(\sigma)^t (A_i D_{\sigma}) \rho(\sigma)
  \\
 &= T(\sigma^{-1}) \psi_{D_{\sigma}} (A_i).
\end{split}
\end{equation}
while eq. (\ref{e19}) implies that 
\begin{equation}
\label{C}
\rho(\sigma)^tB_i \rho(\sigma) - \sum_{j=1}^r \lambda_{ij}(\sigma^{-1}) B_j =
-T(\sigma^{-1}) \psi_{B_\sigma[f]}(A_i).
\end{equation}
For the above computations 
recall that for a $g\times g$ matrix $B$, the map $\psi_B$ is defined by 
\[
\psi_B(A_i)=A_iB+B^t A_i.
\]
Combining now eq.   (\ref{B}) and  (\ref{C}) we have that eq. (\ref{spanR2}) is equivalent to  
\[
T(\sigma^{-1})
\big(
\psi_{D_{\sigma}}(A_i)
\big)
-T(\sigma^{-1}) \psi_{B_\sigma[f]}(A_i)
=\sum_{j=1}^r \mu_{ij}(\sigma) A_j,
\]
thus 
\begin{align}
\big(
\psi_{D_{\sigma}}(A_i)
\big)
-\psi_{B_\sigma[f]}(A_i)
&=\sum_{j=1}^r  T(\sigma) \mu_{ij}(\sigma) A_j
\label{FeqD}
\\
&=
\sum_{j=1}^r  \sum_{\nu=1}^r \mu_{ij}(\sigma) \lambda_{j \nu}(\sigma^{-1}) A_\nu. 
\nonumber \\
&= D^{(1)}(\sigma^{-1})A_i,
\nonumber
\end{align}
where the second equality holds since the action $T$ on $A_1,\ldots,A_r$ is given in terms of the matrix $(\lambda_{i,j})$. Equation (\ref{eq:compat-cond})  and proposition \ref{prop:4compat} is now proved.


Let us note that the restriction of the $T$-action to the generators of the ideal of the relative curve $X_{\Gamma{\prime}}$ corresponds to a lifting of the $\rho^{(1)}$-representation (see also Remark~\ref{rem:Trho1actions}).
In conclusion, this equation expresses a necessary compatibility condition between the representations $\rho$ and $\rho^{(1)}$, which must be satisfied whenever a lift of the action exists.



 \def\cprime{$'$}

\end{document}